\documentclass{article}
\usepackage[margin=1in]{geometry}
\usepackage{amsmath}
\usepackage{amsthm}
\usepackage{amssymb}
\usepackage{array}
\usepackage{booktabs}
\usepackage{graphicx}
\usepackage{hyperref}
\usepackage{fvextra}

\newtheorem{theorem}{Theorem}
\newtheorem{proposition}[theorem]{Proposition}
\newtheorem{lemma}[theorem]{Lemma}
\newtheorem{corollary}[theorem]{Corollary}
\theoremstyle{remark}
\newtheorem{remark}{Remark}

\title{Optimal Extensions of Cross-Sections:\\
Sphere Packings in Dimensions 38 to 43}
\author{Ivan Dorofeev\thanks{\texttt{qft98romanov@gmail.com}} \and
  Xiaoming Sun\thanks{State Key Lab of Processors, Institute of Computing Technology, Chinese Academy of Sciences, Beijing, China. University of Chinese Academy of Sciences, Beijing, China.
  \texttt{sunxiaoming@ict.ac.cn}} \and
  Chengu Wang\thanks{\texttt{wangchengu@gmail.com}}
}
\date{}

\begin{document}

\maketitle

\begin{abstract}
  We improve the best known sphere packings in every dimension from $38$ to
  $43$. Our packings in dimensions $38$ to $42$ come from one chain of cross-sections of the extremal
  even unimodular lattice $P_{48p}$,
  $\sqrt3E_6 \subset \sqrt3E_7 \subset \sqrt3E_8 \subset K_9 \subset K_{10}$,
  whose first three members are cut from the fixed lattice
  $\sqrt3(E_8 \perp E_8)$ of an order-three automorphism; their orthogonal
  complements are lattice packings and set the
  records in dimensions $42$ down to $38$. Each successive section is a
  determinant-minimal extension of its predecessor.
  Our $43$-dimensional packing is an antipode packing: six translates of
  the complement of a $5$-dimensional section.

  We also improve some kissing numbers. Conway and
  Sloane's twelve $1982$ cross-section packings appear never to have had
  theirs computed; we compute them and find that, in dimensions $42$ to $47$,
  they exceed the previously tabulated lower bounds.
  Our chain does better
  in dimensions $40$, $41$ and $42$, and a further antipode packing
  beats the record in dimension $45$.
\end{abstract}

\section{Introduction}
\label{sec:intro}

A \emph{sphere packing} in $\mathbb{R}^n$ is a collection of congruent balls
with disjoint interiors, and its \emph{density} is the fraction of space the
balls cover. Following Conway and Sloane~\cite{conway1999sphere}, we measure packings by their
\emph{center density}
\[
  \delta(\mathcal{P}) \;=\; \frac{\Delta(\mathcal{P})}{V_n},
\]
where $\Delta(\mathcal{P})$ is the density and $V_n$ is the volume of the unit
ball. The sphere packing problem is solved only in dimensions $1$, $2$, $3$
\cite{hales2005proof}, $8$ \cite{viazovska2017sphere} and $24$
\cite{cohn2017sphere}; in every other dimension the state of the art is a
record set by an explicit construction, and we refer to the table maintained
by Cohn~\cite{cohn2026packing}.

Around dimension $48$ the densest packings known come from the extremal even
unimodular lattices of that dimension and their cross-sections. In
\emph{Laminated lattices}, Conway and Sloane~\cite[Corollary~8]{conway1982laminated} exhibited
cross-sections of $P_{48p}$ that set the records in dimensions $39$--$43$,
and those records stood for four decades; the record in dimension $38$ was
held by a lattice of another family, $D_{38}$ of center density $8$
\cite[p.~xxxvii]{conway1999sphere}. The antipode construction of
Conway and Sloane~\cite{conway1996antipode}, which places several translates of a
cross-section over a cluster in its dual, improved dimensions $44$--$47$;
Chen et al.~\cite{chen2025antipode} applied it to deliberately suboptimal cross-sections
and improved dimensions $19$, $20$, $21$, $23$, $44$, $45$ and $47$.

This paper improves all six dimensions from $38$ to $43$
(Table~\ref{tab:results}), and, more to the point, derives them from a
single mechanism rather than six separate tricks.

\begin{table}[ht]
  \centering
  \small
  \setlength{\tabcolsep}{3pt}
  \begin{tabular}{ccccccc}
    \toprule
    & \multicolumn{3}{c}{previously known} & \multicolumn{3}{c}{this paper} \\
    \cmidrule(lr){2-4} \cmidrule(lr){5-7}
    dim & construction & lattice? & $\delta$ & construction & lattice? & $\delta$ \\
    \midrule
    $38$ & $D_{38}$ \cite{conway1999sphere} & yes & $2^3 = 8$
    & Section~\ref{sec:chain} & yes & $2^{-19}3^{31/2}5^{-1} = 9.48\ldots$ \\
    $39$ & cross-section \cite{conway1982laminated} & yes & $2^{-41/2}3^{16}7^{-1/2} = 10.97\ldots$
    & Section~\ref{sec:chain} & yes & $2^{-20}3^{16}5^{-1/2} = 18.35\ldots$ \\
    $40$ & cross-section \cite{conway1982laminated} & yes & $2^{-45/2}3^{17} = 21.77\ldots$
    & Section~\ref{sec:chain} & yes & $2^{-20}3^{16} = 41.05\ldots$ \\
    $41$ & cross-section \cite{conway1982laminated} & yes & $2^{-43/2}3^{17} = 43.54\ldots$
    & Section~\ref{sec:chain} & yes & $2^{-21}3^{17} = 61.57\ldots$ \\
    $42$ & cross-section \cite{conway1982laminated} & yes & $2^{-22}3^{18} = 92.36\ldots$
    & Section~\ref{sec:chain} & yes & $2^{-21}3^{35/2} = 106.65\ldots$ \\
    $43$ & cross-section \cite{conway1982laminated} & yes & $2^{-45/2}3^{19} = 195.94\ldots$
    & Section~\ref{sec:clusters} & no & $2^{-223/2}3^{-21}137^{43/2} = 226.51\ldots$ \\
    \bottomrule
  \end{tabular}
  \caption{Center densities $\delta$. The columns ``lattice?'' say whether
    the packing is a lattice packing; our dimension-$43$ packing is an
    antipode packing of $s = 6$ translates of a lattice
  in~\eqref{eq:master}.}
  \label{tab:results}
\end{table}

\paragraph{The mechanism.}
Let $\Lambda$ be one of the four known extremal even unimodular
$48$-dimensional lattices, so $\det\Lambda = 1$ and the minimal squared norm
is $\mu = 6$. A saturated cross-section $K \subset \Lambda$ of rank $k$ has an
orthogonal complement $L = \Lambda \cap \operatorname{span}(K)^{\perp}$ of
rank $l = 48-k$ with $\det L = \det K$ and minimum at least $6$, and a finite
set $S$ of $s$ points of the dual lattice $K^{*}$ of pairwise squared
distance at most
$\beta$ produces, by the \emph{antipode construction} of Conway and
Sloane~\cite{conway1996antipode}, an $l$-dimensional packing of center density
\begin{equation}
  \label{eq:master}
  \delta \;=\; \frac{s}{2^{l}}\,
  \frac{(6-\beta)^{l/2}}{\sqrt{\det K}} .
\end{equation}
Everything in this paper is a choice of $(K, S)$ in~\eqref{eq:master}, and
the two ways to make $\delta$ large pull against each other: a small
$\det K$ makes $K^{*}$ sparse, which limits $s$, while a cluster large enough
to pay for a bigger determinant needs a $K^{*}$ with many short vectors.

Formula~\eqref{eq:master} and the two bounds that constrain it are due to
Conway and Sloane, and are collected in Section~\ref{sec:prelim}. This
paper adds four things. First, a sharpening of their determinant bound for
one-vector extensions. Write $\rho_{\mathrm{glue}}(K)$ for the largest
distance from a point of $K^{*}$ to $K$, a maximum over the
$\det K$ classes of $K^{*}/K$; we call it the \emph{glue radius}.
Then every rank-$(k+1)$ cross-section containing $K$ has determinant at least
$\det K\,(6-\rho^2_{\mathrm{glue}}(K))$, in \emph{every} integral ambient of
minimum $6$ (Lemma~\ref{lem:covering}). Conway and Sloane control the analogous
quantity by a subcovering radius and then by the ordinary covering radius;
here integrality restricts the relevant projections to $K^{*}$, giving
$\rho_{\mathrm{glue}}(K)\le\rho_{\mathrm{cov}}(K)$, the ordinary covering
radius, and a finite computation of rank $k$. Second, the exact values of $\rho^2_{\mathrm{glue}}$ along the chain
$ \sqrt3E_6 \;\subset\; \sqrt3E_7 \;\subset\; \sqrt3E_8 \;\subset\;
K_9 \;\subset\; K_{10} $
of cross-sections of $P_{48p}$, which prove each one the smallest-determinant
extension of the one before it (Theorem~\ref{th:chain}). Third, the sections
$K_9$ and $K_{10}$ themselves, whose complements are the records in dimensions
$39$ and $38$; the first three sections of the chain come from the fixed
lattice $\sqrt3(E_8\perp E_8)$ of an order-three automorphism of $P_{48p}$ and
give dimensions $42$, $41$ and $40$. Fourth, the dimension-$43$ packing, which
by Corollary~\ref{cor:d43} cannot come from a bare cross-section at the
inherited ambient radius.

Section~\ref{sec:kissing} adds kissing numbers, which
Proposition~\ref{prop:kissing} makes a finite count over the ambient's
minimal vectors. Corollary~8 of Conway and Sloane~\cite{conway1982laminated}
constructs twelve cross-section packings of $P_{48p}$ in dimensions $38$ to
$47$ and states their densities, but their kissing numbers appear never to
have been computed. We reconstruct all twelve in closed form and count them
exactly, and find that in dimensions $42$ to $47$ their kissing numbers exceed
the previously tabulated lower bounds. Our
chain does better still in dimensions $40$, $41$ and $42$, and one further
cross-section, of $P_{48m}$ and $P_{48q}$, beats the record in
dimension $45$ while its density ties the record.

\begin{theorem}
  \label{th:main}
  Sphere packings exist that give
  \[
    \begin{array}{lll}
      \delta^{\mathrm{lat}}_{38} \ge \dfrac{3^{31/2}}{2^{19}\cdot 5} = 9.48\ldots, &
      \delta^{\mathrm{lat}}_{39} \ge \dfrac{3^{16}}{2^{20}\cdot {5}^{1/2}} = 18.35\ldots, &
      \delta^{\mathrm{lat}}_{40} \ge \dfrac{3^{16}}{2^{20}} = 41.05\ldots, \\[3ex]
      \delta^{\mathrm{lat}}_{41} \ge \dfrac{3^{17}}{2^{21}} = 61.57\ldots, &
      \delta^{\mathrm{lat}}_{42} \ge \dfrac{3^{35/2}}{2^{21}} = 106.65\ldots, &
      \delta_{43} \ge \dfrac{137^{43/2}}{2^{223/2}\cdot 3^{21}} = 226.51\ldots
    \end{array}
  \]
  The first five are attained by lattice packings, that is by bare
  cross-section packings~\eqref{eq:bare}, and the last by a union of six
  translates of a lattice (Table~\ref{tab:results}).
\end{theorem}

\begin{table}[ht]
  \centering
  \begin{tabular}{clclc}
    \toprule
    & \multicolumn{2}{c}{previously known} & \multicolumn{2}{c}{this paper} \\
    \cmidrule(lr){2-3} \cmidrule(lr){4-5}
    dim & construction & kissing number & construction & kissing number \\
    \midrule
    $40$ & binary codes \cite{edel1998kissing, brouwer2026tables, cohn2026kissing} & $1{,}064{,}368$
    & cross-section Section~\ref{sec:chain} & $1{,}092{,}000$ \\
    $41$ & binary codes \cite{edel1998kissing, brouwer2026tables, cohn2026kissing} & $1{,}170{,}384$
    & cross-section Section~\ref{sec:chain} & $1{,}324{,}472$ \\
    $42$ & cross-section \cite[Corollary~8]{conway1982laminated} & $1{,}541{,}292$
    & cross-section Section~\ref{sec:chain} & $1{,}792{,}386$ \\
    $45$ & cross-section \cite[Corollary~8]{conway1982laminated} & $6{,}702{,}080$
    & antipode Section~\ref{sec:kissing45} & $7{,}379{,}838$ \\
    \bottomrule
  \end{tabular}
  \caption{Kissing numbers. The
    previously known entries are the largest previously tabulated: in
    dimensions $40$ and $41$ these are kissing configurations built from
    binary codes, not known to extend to packings~\cite{edel1998kissing}; in
    dimensions $42$ and $45$ they are cross-sections of Conway and
    Sloane~\cite[Corollary~8]{conway1982laminated}, constructed without
    computing their kissing numbers, which we compute here
    (Table~\ref{tab:cor8}). Our packings in dimensions $40$, $41$ and $42$ are
    lattices; the dimension-$45$ one is an antipode packing of $s = 4$
  translates of a lattice.}
  \label{tab:kissresults}
\end{table}

\begin{theorem}
  \label{th:kissing}
  There are sphere packings with kissing numbers
  $1{,}092{,}000,\, 1{,}324{,}472,\,  1{,}792{,}386,\, 7{,}379{,}838$
  in dimensions $40$, $41$, $42$ and $45$
  (Table~\ref{tab:kissresults}). The first three are lattice packings; the
  last is a union of four translates of a lattice.
\end{theorem}

\section{Preliminaries}
\label{sec:prelim}

\subsection{Cross-sections}
\label{sec:sections}

Let $\Lambda$ be an $n$-dimensional lattice of minimal squared norm $\mu$ and
let $\mathbb{R}^n = U \oplus V$ be an orthogonal decomposition with
$\dim U = k$, $\dim V = l = n-k$. Write
\[
  K = \Lambda \cap U, \qquad L = \Lambda \cap V, \qquad
  M = \pi_U(\Lambda),
\]
where $\pi_U$, $\pi_V$ are the orthogonal projections, and assume $K$ has
rank $k$. A sublattice $K \subseteq \Lambda$ is \emph{saturated} (or
primitive) when $\Lambda/K$ is torsion-free; $\Lambda \cap U$ is always
saturated, and conversely a saturated sublattice is the intersection of
$\Lambda$ with its own span. For $\Lambda$ unimodular, discriminant gluing
identifies the discriminant groups of $K$ and $L$ anti-isometrically
\cite[Prop.~1.9.8]{martinet2003perfect}, \cite[Thm.~4]{conway1982laminated},
so that
\begin{equation}
  \label{eq:detequal}
  \det L = \det K = 1/\det M .
\end{equation}
We write $K^{*}$ and $L^{*}$ for the duals of $K$ in $U$ and of $L$ in $V$.
The next fact is used throughout; it needs $L$ saturated, which is automatic,
and not $K$ saturated.

\begin{lemma}
  \label{lem:projdual}
  If $\Lambda$ is unimodular and $L = \Lambda \cap V$, then
  $\pi_V(\Lambda) = L^{*}$. Likewise $\pi_U(\Lambda) = K^{*}$ when
  $K = \Lambda \cap U$.
\end{lemma}

\begin{proof}
  For $x \in \Lambda$ and $\ell \in L$ we have
  $\langle \pi_V(x), \ell\rangle = \langle x, \ell\rangle \in \mathbb{Z}$, so
  $\pi_V(\Lambda) \subseteq L^{*}$. Conversely let $y \in L^{*}$. Since
  $\Lambda/L$ is torsion-free it is free, so $L$ is a direct summand of
  $\Lambda$ and the homomorphism $\ell \mapsto \langle y, \ell \rangle$
  extends to a homomorphism $\Lambda \to \mathbb{Z}$. As $\Lambda$ is
  unimodular, that homomorphism is $\langle x, \cdot\rangle$ for some
  $x \in \Lambda$, and then $\pi_V(x) - y$ is orthogonal to $L$ and lies in
  $V = \operatorname{span}(L)$, hence vanishes.
\end{proof}

\subsection{The antipode construction}

Let $S = \{u_0 = 0, u_1, \ldots, u_{s-1}\} \subset M$ have pairwise squared
distances at most $\beta$ for some $\beta < \mu$. The \emph{antipode packing}
is
\[
  \mathcal{A}(S) \;=\; \{\, \pi_V(w) : w \in \Lambda,\ \pi_U(w) \in S \,\},
\]
a union of $s$ translates of $L$ packed with spheres of radius
$\tfrac12\sqrt{\mu-\beta}$.

\begin{theorem}[Conway and Sloane \cite{conway1996antipode}]
  \label{th:antipode}
  Drawing spheres of radius $\rho = \tfrac12\sqrt{\mu-\beta}$ around the
  points of $\mathcal{A}(S)$ gives an $l$-dimensional sphere packing of center
  density
  \[
    \delta(\mathcal{A}(S)) \;=\; \frac{s\,\rho^{l}}{\sqrt{\det L}} .
  \]
\end{theorem}

For unimodular $\Lambda$ we have $\det L = \det K$ by~\eqref{eq:detequal}, so
this is $s(\mu-\beta)^{l/2}2^{-l}(\det K)^{-1/2}$, which
is~\eqref{eq:master} when $\mu = 6$. The case $s = 1$, $\beta = 0$ uses the
lattice $L$ with spheres of radius $\sqrt{\mu}/2$, the radius inherited from
the ambient minimum. We call this the \emph{bare cross-section packing}; five
of our six results are of this kind, with center density
\begin{equation}
  \label{eq:bare}
  \delta_{\mathrm{bare}}(L) \;=\;
  \frac{(\mu/4)^{l/2}}{\sqrt{\det K}}
  \;=\; \frac{(3/2)^{(48-k)/2}}{\sqrt{\det K}}
  \qquad (\mu = 6,\ n = 48).
\end{equation}
If $\min(L)>\mu$, then $L$ regarded intrinsically as a lattice can of course
support larger spheres. Throughout this paper, ``bare cross-section packing''
means the construction at the inherited ambient radius $\sqrt{\mu}/2$.

The kissing number of an antipode packing is equally explicit: contacts come
from minimal vectors of $\Lambda$ lying in the fibers of $\pi_U$ over the
difference set of $S$.

\begin{proposition}
  \label{prop:kissing}
  In the packing $\mathcal{A}(S)$, every sphere in the translate indexed by
  $u_i$ touches exactly
  \[
    N_i \;=\; \sum_{\substack{j \ne i \\ |u_i - u_j|^2 = \beta}}
    \#\{\, w \in \Lambda : |w|^2 = \mu,\ \pi_U(w) = u_j - u_i \,\}
  \]
  other spheres, and the kissing number of $\mathcal{A}(S)$ is $\max_i N_i$.
\end{proposition}

\begin{proof}
  Fix a center $\pi_V(w_0)$ with $\pi_U(w_0) = u_i$. Every other center is
  $\pi_V(w)$ with $\pi_U(w) = u_j$ for some $j$, and $w \mapsto w - w_0$ is a
  bijection between such centers and the vectors $v \in \Lambda$ with
  $\pi_U(v) = u_j - u_i$ (for $j = i$, excluding $v = 0$). The squared distance
  of the two centers is
  \[
    |\pi_V(w - w_0)|^2 \;=\; |w - w_0|^2 - |u_j - u_i|^2 \;\ge\; \mu - \beta
    \qquad (w \ne w_0),
  \]
  since $|w - w_0|^2 \ge \mu$ and $|u_j - u_i|^2 \le \beta$ (for $j = i$ the
  bound is $\mu > \mu - \beta$). Equality holds if and only if $w - w_0$ is a
  minimal vector of $\Lambda$ and $|u_j - u_i|^2 = \beta$; in this case the
  two spheres of radius $\tfrac12\sqrt{\mu-\beta}$ are tangent. Summing over
  $j$ gives $N_i$, which does not depend on the choice of $w_0$.
\end{proof}

\begin{remark}
  \label{rem:bare}
  For $s = 1$ the packing $\mathcal{A}(S)$ is the lattice packing of $L$
  itself, the contact distance is $\mu$, and the same argument gives
  $N_0 = \#\{\, w \in \Lambda : |w|^2 = \mu,\ \pi_U(w) = 0 \,\}$: the kissing
  number of a bare cross-section packing is the number of minimal vectors of
  $\Lambda$ that lie in $L$. Section~\ref{sec:kissing} uses this count.
\end{remark}

A convenient sufficient condition for a large cluster, which our
dimension-$43$ packing satisfies, takes $S$ to be the origin together with an
entire basis of $M$:

\begin{theorem}[Chen, Hu, Li, Wang and Wu \cite{chen2025antipode}]
  \label{th:simplex}
  Let $\Lambda$ be an $n$-dimensional unimodular lattice of minimum $\mu$,
  let $K$ be a saturated cross-section of rank $k$ with Gram matrix
  $\mathbf{K}$, and put $\mathbf{M} = \mathbf{K}^{-1}$. If
  $0 < \beta < \mu$ satisfies
  \[
    \mathbf{M}_{i,i} \le \beta
    \qquad\text{and}\qquad
    \mathbf{M}_{i,i} + \mathbf{M}_{j,j} - 2\mathbf{M}_{i,j} \le \beta
    \qquad (1 \le i < j \le k),
  \]
  then there is an $(n-k)$-dimensional packing of center density
  $(k+1)\sqrt{\det \mathbf{M}}\,((\mu-\beta)/4)^{(n-k)/2}$.
\end{theorem}

Three center densities appear below. Write $\delta_d$ for the largest center
density of a packing in dimension $d$, and $\delta^{\mathrm{lat}}_{d}$ for the
largest center density of a lattice packing, each taken over all packings
whether or not one is known; so
$\delta^{\mathrm{lat}}_{d} \le \delta_d$, and
$\delta^{\mathrm{lat}}_{d} = (\gamma_d/4)^{d/2}$ with $\gamma_d$ the Hermite
constant. A cross-section is a lattice, so what bounds it is
$\delta^{\mathrm{lat}}$, not $\delta$, and the distinction is not vacuous:
among the ranks used below, the densest packing known in dimension $10$ is not
a lattice \cite{cohn2026packing}. Neither optimum is known in every rank we
need, so write $\underline{\delta}^{\mathrm{lat}}_{d}$ for the largest center
density of a lattice packing that is \emph{known}, tabulated by Nebe and
Sloane~\cite{nebe2026density}; then
$\underline{\delta}^{\mathrm{lat}}_{d} \le \delta^{\mathrm{lat}}_{d}$, with
equality exactly when the densest rank-$d$ lattice has been found.

\subsection{The ceiling of a bare cross-section}
\label{sec:ceiling}

A cross-section of $\Lambda$ is a lattice of minimum at least $\mu = 6$, and
that alone limits how dense its complement can be. The argument is that of
Conway and Sloane, who observed that the densest $k$-dimensional and
$(48-k)$-dimensional sections of a unimodular lattice have the same
determinant \cite[Corollary~5]{conway1982laminated}; with the classification
of the densest lattices in dimensions at most $8$
\cite[Theorem~2]{conway1982laminated} this bounds one in terms of the other.
They drew the conclusion in dimensions $43$ and above, over even unimodular
ambients \cite[Corollary~8]{conway1982laminated}; we need it below dimension
$43$ as well, where their own sections do not attain it
(Corollary~\ref{cor:optimal}), and over every unimodular ambient.

\begin{proposition}
  \label{prop:ceiling}
  Let $K$ be a cross-section of rank $k$ of a unimodular $48$-dimensional
  lattice of minimum $6$, and let $\gamma_k$ be the Hermite constant. Then
  \[
    \det K \;\ge\; \left(\frac{6}{\gamma_k}\right)^{k},
  \]
  and consequently the bare cross-section packing~\eqref{eq:bare} of its
  complement has
  \begin{equation}
    \label{eq:ceiling}
    \delta_{\mathrm{bare}}
    \;\le\;
    \delta^{\mathrm{lat}}_{k}\,(3/2)^{24-k}.
  \end{equation}
\end{proposition}

\begin{proof}
  $K$ is a rank-$k$ lattice of minimum $m \ge 6$, so
  $6/(\det K)^{1/k} \le m/(\det K)^{1/k} \le \gamma_k$ by the definition of
  the Hermite constant, which is the determinant inequality. Substituting it
  into~\eqref{eq:bare} gives the density.
\end{proof}

The bound depends on neither the choice of $48$-dimensional lattice nor
extremality: the same argument over an ambient of minimum $\mu \le 6$ only
lowers the right-hand side, and an odd unimodular lattice of rank $48$ also
has minimum at most $6$ by the shadow bound of Rains and
Sloane~\cite{rains1998shadow}, so the ceiling holds over every unimodular
ambient of rank $48$.

The densest lattice packing has been established in the published literature
in every dimension up to $8$, so for $k\le8$ the bound is effective:
$\gamma_5^5=8$, $\gamma_6^6=64/3$, $\gamma_7^7=64$ and
$\gamma_8^8=256$ \cite{korkine1877formes,blichfeldt1935minimum}, attained
by $D_5$, $E_6$, $E_7$ and $E_8$, unique up to similarity
\cite[Ch.~6]{martinet2003perfect}.\footnote{A recent preprint of Dutour
  Sikiri\'c and van Woerden~\cite{dutoursikiric2025dimension9} reports a
  complete Voronoi enumeration in dimension $9$ and concludes that $\Lambda_9$
  is optimal, equivalently $\gamma_9=2$. If this computation is correct, then
  the $k=9$ row of Table~\ref{tab:ceiling} is also unconditional:
  $\det K\ge3^9=19683$, and the possible bare-improvement window in dimension
  $39$ is exactly $19683\le\det K<21870$. Because this result is available only
  as a preprint at the time of writing, we do not use it as an established
theorem in the present paper.}
Those are the ranks at which this paper uses published optimality results.
Section~\ref{sec:clusters} evaluates the bound rank by rank.

\section{A chain of optimal extensions}
\label{sec:chain}

\subsection{The seed: an order-three fixed lattice}
\label{sec:seed}

Section~\ref{sec:ceiling} says what to look for in ranks $6$, $7$, $8$: a
section isometric to $\sqrt3E_6$, $\sqrt3E_7$ or $\sqrt3E_8$ would attain the
ceiling exactly. Symmetry produces all three at once.

Let $\Lambda = P_{48p}$, the extremal even unimodular lattice built from the
Pless symmetry code \cite{pless1972symmetry, conway1999sphere}, so
$\det \Lambda = 1$ and $\mu(\Lambda) = 6$. For an automorphism $\sigma$ of
$\Lambda$ let $\operatorname{Fix}_\Lambda(\sigma) = \Lambda \cap
\ker(\sigma - 1)$.

\begin{proposition}[Nebe \cite{nebe2013automorphisms}]
  \label{prop:nebe}
  There is an automorphism $\sigma \in \operatorname{Aut}(P_{48p})$ of order
  $3$ with
  $\operatorname{Fix}_{P_{48p}}(\sigma) \cong \sqrt3(E_8 \perp E_8)$.
\end{proposition}

This is the construction-specific input from the literature: the proof of
Theorem~3.8 of Nebe~\cite{nebe2013automorphisms} exhibits an order-three
$\sigma$ with $\operatorname{Fix}(\sigma) \cong \sqrt3(E_8\perp E_8)$ in
$P_{48p}$, and Proposition~4.8 there classifies the order-three automorphisms
of extremal $48$-dimensional lattices. Here
$\sqrt3R$ means the lattice $R$ with its bilinear form multiplied by $3$, so
$\det(\sqrt3R) = 3^{\operatorname{rank}R}\det R$ and the minimum of $\sqrt3R$
is three times that of $R$.

The two lemmas below are standard, and are proved only to keep the
determinants self-contained: invariant lattices are saturated
\cite[\S2]{hohn2016fixed} (see \cite{nikulin1980integral} for saturated
sublattices in general), and a subsystem of simple roots spans a saturated
sublattice. It is their conjunction with Proposition~\ref{prop:nebe} that
matters, putting $\sqrt3E_6$, $\sqrt3E_7$ and $\sqrt3E_8$ inside an extremal
$48$-dimensional lattice, where~\eqref{eq:bare} turns them into packings.

\begin{lemma}
  \label{lem:fixedprim}
  $F = \operatorname{Fix}_\Lambda(\sigma)$ is saturated in $\Lambda$.
\end{lemma}

\begin{proof}
  $F = \Lambda \cap \ker(\sigma - \operatorname{id})$ is the intersection
  of $\Lambda$ with a linear subspace, hence saturated by
  Section~\ref{sec:sections}. Directly: if $x \in \Lambda$ and $mx \in F$
  for some nonzero integer $m$, then $m(\sigma x - x) = \sigma(mx) - mx = 0$;
  a lattice is torsion-free, so $\sigma x = x$ and $x \in F$.
\end{proof}

\begin{lemma}
  \label{lem:nested}
  $\Lambda$ contains saturated sublattices
  $K_6 \cong \sqrt3E_6 \subset K_7 \cong \sqrt3E_7 \subset K_8 \cong
  \sqrt3E_8$, of determinants $3^7 = 2187$, $2\cdot3^7 = 4374$ and
  $3^8 = 6561$.
\end{lemma}

\begin{proof}
  Write $F = F_1 \perp F_2$ with $F_i \cong \sqrt3E_8$; each $F_i$ is a direct
  summand of the free abelian group $F$. Choose a simple-root basis of $F_1$
  containing the nested Dynkin diagrams $E_6 \subset E_7 \subset E_8$. A
  subset of a $\mathbb{Z}$-basis spans a saturated subgroup, so the scaled
  copies of $E_6$ and $E_7$ are saturated in $F_1$. Saturation is transitive,
  so Lemma~\ref{lem:fixedprim} gives all three saturated embeddings in
  $\Lambda$. Finally $\det E_6 = 3$, $\det E_7 = 2$ and $\det E_8 = 1$.
\end{proof}

Applying~\eqref{eq:bare} to $K_6$, $K_7$, $K_8$ gives the first three
records of Table~\ref{tab:results}:
\[
  \delta_{42} \ge \frac{(3/2)^{21}}{\sqrt{3^7}} = \frac{3^{35/2}}{2^{21}},
  \qquad
  \delta_{41} \ge \frac{(3/2)^{41/2}}{\sqrt{2\cdot3^7}} = \frac{3^{17}}{2^{21}},
  \qquad
  \delta_{40} \ge \frac{(3/2)^{20}}{\sqrt{3^8}} = \frac{3^{16}}{2^{20}} .
\]
Each is exactly the ceiling~\eqref{eq:ceiling} of its rank, so by
Proposition~\ref{prop:ceiling}:

\begin{corollary}
  \label{cor:optimal}
  The packings of dimensions $42$, $41$ and $40$ above are optimal among bare
  cross-section packings: no bare cross-section packing~\eqref{eq:bare} of
  any unimodular $48$-dimensional lattice has larger center density in the
  same dimension. Equivalently, $K_6$, $K_7$ and $K_8$ have the smallest
  determinant among all cross-sections of rank $6$, $7$ and $8$ of such
  lattices, attaining the floor $(6/\gamma_k)^{k}$ of
  Proposition~\ref{prop:ceiling}, and the sections attaining it are unique
  up to isometry.
\end{corollary}

Conway and Sloane claimed optimality only from dimension $43$ upwards
\cite[Corollary~8]{conway1982laminated}. Their sections in dimensions $42$,
$41$ and $40$ are $\sqrt3D_6$, $\sqrt3D_7$ and $S_8$, of determinants
$4\cdot3^6$, $4\cdot3^7$ and $32\cdot3^6$, none of which meets
Proposition~\ref{prop:ceiling}: what was missing was not the bound but a
section attaining it. Corollary~\ref{cor:optimal} extends their assertion down
to dimension $40$.

In the basis of simple roots the three Gram matrices are the leading blocks
of
\[
  3\,C(E_8) \;=\;
  \begin{bmatrix}
    6 &  0 & -3 &  0 &  0 &  0 &  0 &  0 \\
    0 &  6 &  0 & -3 &  0 &  0 &  0 &  0 \\
    -3 &  0 &  6 & -3 &  0 &  0 &  0 &  0 \\
    0 & -3 & -3 &  6 & -3 &  0 &  0 &  0 \\
    0 &  0 &  0 & -3 &  6 & -3 &  0 &  0 \\
    0 &  0 &  0 &  0 & -3 &  6 & -3 &  0 \\
    0 &  0 &  0 &  0 &  0 & -3 &  6 & -3 \\
    0 &  0 &  0 &  0 &  0 &  0 & -3 &  6
  \end{bmatrix},
\]
$C(E_8)$ being the Cartan matrix; the leading $6\times6$ and $7\times7$
blocks are $3C(E_6)$ and $3C(E_7)$, of determinants $2187$ and $4374$.

\begin{remark}
  \label{rem:ambients}
  The construction is not specific to $P_{48p}$. Realizing $\sigma$ as a
  signed coordinate permutation of the ternary code frame (for $P_{48p}$
    the element $x \mapsto -1/(x+1)$ of $\mathrm{PSL}_2(23)$ acting on the
    Pless symmetry code, for $P_{48q}$ the corresponding element of
  $\mathrm{PSL}_2(47)$ on the extended quadratic residue code) and, for
  $P_{48n}$, as the product of three generators of the automorphism group
  listed in the catalogue entry \texttt{CQ48a} \cite{nebe2026catalogue}, we
  obtain the same fixed lattice $\sqrt3(E_8\perp E_8)$, hence the same
  sections, in three of the four known extremal lattices. Whether the fourth,
  $P_{48m}$, has such a section we do not know. An order-three action on
  $\mathbb{R}^{48}$ with $f$ trivial and $t$ faithful two-dimensional summands
  has $48 = f + 2t$ and trace $f - t$, so a fixed lattice of rank $16$ requires
  trace $0$; an unverified enumeration of order-three elements generated by
  the catalogue's generators for $\operatorname{Aut}(P_{48m})$ produced only
  trace $-24$, that is, only fixed-point-free elements, but we have not
  certified that computation and nothing in this paper depends on it.
\end{remark}

\subsection{Extension and the glue radius}
\label{sec:growing}

We now push past rank $8$. The determinant of a one-vector extension is
governed by a projection, and integrality restricts that projection to a dual-lattice coset.

\begin{proposition}
  \label{prop:extend}
  Let $K = \Lambda \cap U$ be a saturated cross-section of rank $k$ of a
  unimodular lattice $\Lambda$, and $L = \Lambda \cap V$ with $V = U^{\perp}$.
  \begin{enumerate}
    \item Every saturated cross-section of rank $k+1$ containing $K$ is
      $K + \mathbb{Z}v$ for some $v \in \Lambda$.
    \item $\det(K + \mathbb{Z}v) = \det K \cdot |\pi_V(v)|^2$.
    \item The smallest determinant of a rank-$(k+1)$ cross-section containing
      $K$ is $\det K \cdot \mu(L^{*})$, where $\mu$ denotes the minimum.
  \end{enumerate}
\end{proposition}

\begin{proof}
  (1) If $W \supseteq K$ is a saturated section of rank $k+1$ then $W/K$ is
  torsion-free of rank $1$, hence infinite cyclic; any preimage $v$ of a
  generator gives $W = K + \mathbb{Z}v$.
  (2) In a basis of $K$ followed by $v$, the Gram matrix of
  $K + \mathbb{Z}v$ has the Gram matrix of $K$ as its leading block and
  $|v|^2$ in the corner; its Schur complement is
  $|v|^2 - |\pi_U(v)|^2 = |\pi_V(v)|^2$.
  (3) By Lemma~\ref{lem:projdual}, $\pi_V$ maps $\Lambda$ onto $L^{*}$, so
  the values $|\pi_V(v)|^2$ for $v \in \Lambda$ with $\pi_V(v) \ne 0$ are
  exactly the nonzero squared norms of $L^{*}$, and the smallest is
  $\mu(L^{*})$. If the resulting $K + \mathbb{Z}v$ is not saturated, its
  saturation is a section of the same rank and strictly smaller determinant
  containing $K$, so the minimum is attained by a saturated section.
\end{proof}

Parts (2) and (3) are the cross-section form of a computation of Conway and
Sloane. Writing $\Lambda_n$ as a union of layers of a sublattice
$\Lambda_{n-1}$, their equation~(2) reads $X_n = \gamma_{n-1}X_{n-1}$, with
$X$ a determinant and $\gamma_{n-1}$ the minimum of the projection of
$\Lambda_n$ orthogonal to $\Lambda_{n-1}$ \cite[\S II]{conway1982laminated}.
They grow a lattice; we grow a section inside a fixed ambient, where
Lemma~\ref{lem:projdual} identifies the projection as $L^{*}$.

Part (3) replaces a search through $\Lambda$ by one shortest-vector
computation, but in the complement, of dimension $39$ or more. The next lemma
moves the question into rank $k$, where a finite enumeration settles it, and
removes the dependence on the ambient.

\begin{lemma}
  \label{lem:covering}
  Let $\Lambda$ be an integral lattice of minimum $\mu$, let $K\subseteq
  \Lambda$ have rank $k$ and span $U$, put $V=U^{\perp}$, and set
  \[
    \rho_{\mathrm{glue}}(K)
    \;=\;
    \max_{y\in K^{*}} \operatorname{dist}(y,K),
  \]
  a maximum over the finitely many classes of $K^{*}/K$. We call
  $\rho_{\mathrm{glue}}(K)$ the \emph{glue radius}: it is the largest norm of
  a minimal glue vector of $K$ \cite[Ch.~4]{conway1999sphere}. Then every
  $v\in\Lambda$ with $\pi_V(v)\ne0$ satisfies
  \[
    |\pi_V(v)|^2 \;\ge\; \mu-\rho^2_{\mathrm{glue}}(K).
  \]
  Moreover
  $\rho_{\mathrm{glue}}(K)\le\rho_{\mathrm{cov}}(K)$, where
  $\rho_{\mathrm{cov}}(K)$ is the ordinary covering radius of $K$.
\end{lemma}

\begin{proof}
  Since $\Lambda$ is integral, $\langle v,x\rangle\in\mathbb{Z}$ for all
  $x\in K$, so $\pi_U(v)\in K^{*}$. Choose $x\in K$ nearest to
  $\pi_U(v)$; then
  $|\pi_U(v)-x|^2\le\rho^2_{\mathrm{glue}}(K)$. The vector $v-x$ lies in
  $\Lambda$ and is nonzero because $\pi_V(v-x)=\pi_V(v)\ne0$, so
  \[
    \mu
    \;\le\;
    |v-x|^2
    \;=\;
    |\pi_V(v)|^2+|\pi_U(v)-x|^2
    \;\le\;
    |\pi_V(v)|^2+\rho^2_{\mathrm{glue}}(K).
  \]
  Finally, the ordinary covering radius maximizes
  $\operatorname{dist}(y,K)$ over all $y\in U$, while the glue radius
  maximizes it only over $y\in K^{*}$, hence
  $\rho_{\mathrm{glue}}(K)\le\rho_{\mathrm{cov}}(K)$.
\end{proof}

The argument is again theirs: equation~(1) of
\cite[\S II]{conway1982laminated} reads $\gamma_{n-1}\ge4-h_{n-1}^2$, with
$h_{n-1}$ the distance from the extension vector to the layer lattice, a
quantity they call the \emph{subcovering radius} and then bound by the
ordinary covering radius. In the present integral setting, $\pi_U(v)$ is
restricted to $K^{*}$ rather than being an arbitrary point of $U$. Thus the
intrinsic glue radius gives a bound no weaker than the ordinary
covering-radius bound and is computable by a finite enumeration of $K^{*}/K$.
The difference decides the question at hand. The ordinary covering radius of
$\sqrt3E_8$ is $\sqrt3$, so bounding by it gives only
$\det\ge3^8(6-3)=19683$, which leaves dimension $39$ open, whereas
$\rho^2_{\mathrm{glue}}(\sqrt3E_8)=8/3$ gives
$3^8(6-8/3)=21870$, and $21870$ is attained.

\subsection{The chain}

Combining Proposition~\ref{prop:extend}(2) with Lemma~\ref{lem:covering}:
every rank-$(k+1)$ cross-section containing $K$ has determinant at least
$\det K\,(\mu - \rho^2_{\mathrm{glue}}(K))$, in any integral ambient of minimum $\mu$.
Table~\ref{tab:chain} evaluates $\rho^2_{\mathrm{glue}}$ for the sections at hand. Each row's
bound is \emph{attained} by the next row's section, so each step of the chain
is optimal.

\begin{table}[ht]
  \centering
  \begin{tabular}{lrccrll}
    \toprule
    $K$ & $\det K$ & $\rho^2_{\mathrm{glue}}(K)$ & $6-\rho^2_{\mathrm{glue}}(K)$
    & $\det K\,(6-\rho^2_{\mathrm{glue}}(K))$ & attained by & complement \\
    \midrule
    $\sqrt3E_6$ & $2187$ & $4$ & $2$ & $4374$ & $\sqrt3E_7$ & $\dim 41$, $61.57\ldots$ \\
    $\sqrt3E_7$ & $4374$ & $9/2$ & $3/2$ & $6561$ & $\sqrt3E_8$ & $\dim 40$, $41.05\ldots$ \\
    $\sqrt3E_8$ & $6561$ & $8/3$ & $10/3$ & $21870$ & $K_9$ & $\dim 39$, $18.35\ldots$ \\
    $K_9$ & $21870$ & $7/2$ & $5/2$ & $54675$ & $K_{10}$ & $\dim 38$, $9.48\ldots$ \\
    $K_{10}$ & $54675$ & $34/9$ & $20/9$ & $121500$ & & $\dim 37$, $\le 5.19\ldots$ \\
    \bottomrule
  \end{tabular}
  \caption{The chain of optimal extensions. By Lemma~\ref{lem:covering} every
    rank-$(k+1)$ cross-section containing $K$ has determinant at least
    $\det K\,(6-\rho^2_{\mathrm{glue}}(K))$, in any integral ambient of minimum $6$,
    where $k = \operatorname{rank} K$; the fifth column is this floor, and the
    sixth shows that each bound is attained by the next section of the chain,
    so every step is optimal. The values of $\rho^2_{\mathrm{glue}}$ are exact, obtained by
    enumerating all $\det K$ classes of $K^{*}/K$; the maximum is attained by
  $2$, $1$, $1920$, $561$ and $324$ classes respectively.}
  \label{tab:chain}
\end{table}

\begin{theorem}
  \label{th:chain}
  Let $\Lambda$ be a unimodular $48$-dimensional lattice of minimum $6$
  containing a saturated $\sqrt3E_8$. Then:
  \begin{enumerate}
    \item every rank-$9$ cross-section of $\Lambda$ containing that
      $\sqrt3E_8$ has determinant at least $21870$;
    \item in $P_{48p}$, $P_{48q}$ and $P_{48n}$ the value $21870$ is attained
      by a saturated section $K_9$, whose complement is a $39$-dimensional
      lattice packing of center density
      $2^{-20}3^{16}5^{-1/2} = 18.35\ldots$;
    \item every rank-$10$ cross-section containing that $K_9$ has determinant
      at least $54675$, attained by a saturated $K_{10}$ whose complement is
      a $38$-dimensional lattice packing of center density
      $2^{-19}3^{31/2}5^{-1} = 9.48\ldots$;
    \item every rank-$11$ cross-section containing $K_{10}$ has determinant at
      least $121500$, so its complement has center density at most
      $5.19\ldots$, below the record in dimension $37$.
  \end{enumerate}
\end{theorem}

\begin{proof}
  The lower bounds are Lemma~\ref{lem:covering} with the values of $\rho^2_{\mathrm{glue}}$
  in Table~\ref{tab:chain}, computed exactly by enumerating the classes of
  $K^{*}/K$ (Section~\ref{sec:search}); the attainments are the certified
  sections of Appendix~\ref{app:certificates}, and the densities
  are~\eqref{eq:bare}. For (4), $(3/2)^{37/2}/\sqrt{121500} = 5.19\ldots$
  while the dimension-$37$ record is $2^{5/2} = 5.65\ldots$.
\end{proof}

The Gram matrices of the two new sections, in the bases of
Appendix~\ref{app:certificates}, are
\[
  \mathbf{K}_9 \;=\;
  \begin{bmatrix}
    6 & -3 & -3 & -3 & -3 & -3 & -3 &  3 & -2 \\
    -3 &  6 &  3 &  3 &  3 &  3 &  3 &  0 &  1 \\
    -3 &  3 &  6 &  0 &  3 &  0 &  0 & -3 &  3 \\
    -3 &  3 &  0 &  6 &  3 &  3 &  3 &  0 & -1 \\
    -3 &  3 &  3 &  3 &  6 &  3 &  3 & -3 &  1 \\
    -3 &  3 &  0 &  3 &  3 &  6 &  3 &  0 &  0 \\
    -3 &  3 &  0 &  3 &  3 &  3 &  6 &  0 &  1 \\
    3 &  0 & -3 &  0 & -3 &  0 &  0 &  6 & -1 \\
    -2 &  1 &  3 & -1 &  1 &  0 &  1 & -1 &  6
  \end{bmatrix},
\]
\[
  \mathbf{K}_{10} \;=\;
  \begin{bmatrix}
    6 & -3 & -3 & -3 & -3 & -3 & -3 &  3 & -2 &  1 \\
    -3 &  6 &  3 &  3 &  3 &  3 &  3 &  0 &  1 & -1 \\
    -3 &  3 &  6 &  0 &  3 &  0 &  0 & -3 &  3 &  1 \\
    -3 &  3 &  0 &  6 &  3 &  3 &  3 &  0 & -1 &  0 \\
    -3 &  3 &  3 &  3 &  6 &  3 &  3 & -3 &  1 & -1 \\
    -3 &  3 &  0 &  3 &  3 &  6 &  3 &  0 &  0 & -2 \\
    -3 &  3 &  0 &  3 &  3 &  3 &  6 &  0 &  1 & -2 \\
    3 &  0 & -3 &  0 & -3 &  0 &  0 &  6 & -1 &  1 \\
    -2 &  1 &  3 & -1 &  1 &  0 &  1 & -1 &  6 &  2 \\
    1 & -1 &  1 &  0 & -1 & -2 & -2 &  1 &  2 &  6
  \end{bmatrix},
\]
of determinants $21870 = 2\cdot3^7\cdot5$ and $54675 = 3^7\cdot5^2$. The two
bases are nested: the leading $8\times 8$ block of each is a Gram matrix of
$\sqrt3E_8$, and $\mathbf{K}_9$ is the leading $9\times 9$ block of
$\mathbf{K}_{10}$, so the inclusions
$\sqrt3E_8 \subset K_9 \subset K_{10}$ of Theorem~\ref{th:chain} can be read
off the matrices rather than taken on trust. Neither is a
rescaled root lattice, since $21870/3^9$ and $54675/3^{10}$ are not integers;
their discriminant groups are
$(\mathbb{Z}/3)^6\times\mathbb{Z}/30$ and
$(\mathbb{Z}/3)^4\times\mathbb{Z}/15\times\mathbb{Z}/45$, and they have $258$
and $294$ minimal vectors.

\begin{remark}
  \label{rem:e8component}
  Concretely, since $K_8$ is saturated, Lemma~\ref{lem:projdual} gives
  $\pi_U(\Lambda) = K_8^{*} = \tfrac1{\sqrt3}E_8$, so every $v \in \Lambda$
  has $\pi_U(v) = e/\sqrt3$ with $e \in E_8$, and
  $|\pi_V(v)|^2 = |v|^2 - |e|^2/3$: a \emph{larger} $E_8$-component means a
  \emph{smaller} determinant. For a minimal vector $v$ and every root
  $\rho$ of $E_8$ we have $|v - \sqrt3\rho|^2 \ge 6$, i.e.\ $(e,\rho) \le 3$,
  and an enumeration of the vectors of $E_8$ of norm at most
  $|e|^2 \le 3|v|^2 = 18$ shows that this forces $|e|^2 \le 8$. The bound is
  attained: of the $52{,}416{,}000$ minimal vectors of $P_{48p}$, exactly
  $272{,}160$ have $|e|^2 = 8$, and each of them extends $K_8$ to a section
  of determinant $3^8(6 - 8/3) = 21870$. In the Leech lattice the same
  computation gives only $(e,\rho) \le 2$ and no comparable gain, which is
  why this route is special to minimum $6$.
\end{remark}

\begin{remark}
  \label{rem:noclusters}
  For $K_9$ and $K_{10}$, our exhaustive cluster maximization over the duals,
  with no cap on the number of dual points or on the clique search and with
  $\beta\le1$, returns $s=1$. Thus no cluster with $\beta\le1$ improves the
  bare packing. We do not infer anything here about clusters of diameter
  $\beta>1$; such clusters may contain additional dual points and are not
  ruled out by this computation.
\end{remark}

\section{Dimension 43: beyond the bare cross-section}
\label{sec:clusters}

Evaluating Proposition~\ref{prop:ceiling} rank by rank says how far the
bare route reaches. In rank $5$ it stops short of dimension $43$.

\begin{corollary}
  \label{cor:d43}
  No bare cross-section of any unimodular $48$-dimensional lattice gives a
  $43$-dimensional packing of center density above
  $2^{-45/2}3^{19} = 195.94\ldots$, and that value is attained by
  the section $\sqrt3D_5$ of determinant $972$. In particular the
  $43$-dimensional record of Table~\ref{tab:results} cannot be obtained from
  a bare cross-section. This is the second assertion of
  Corollary~8 of Conway and Sloane~\cite{conway1982laminated}, stated there
  for every even unimodular $48$-dimensional lattice of minimum $6$ and for
  all dimensions from $43$ upwards.
\end{corollary}

\begin{proof}
  Proposition~\ref{prop:ceiling} with $k = 5$ and $\gamma_5^5 = 8$ gives
  $\det K \ge 6^5/8 = 972$ and
  $\delta_{\mathrm{bare}} \le (3/2)^{43/2}/\sqrt{972} = 2^{-45/2}3^{19}$,
  with equality only for $\sqrt3D_5$, which is the
  dimension-$43$ section of Conway and Sloane~\cite[Corollary~8]{conway1982laminated}.
\end{proof}

\begin{table}[ht]
  \centering
  \begin{tabular}{@{}cccc>{\raggedright\arraybackslash}p{0.38\textwidth}@{}}
    \toprule
    $k$ & $48-k$ & densest known lattice & bound / benchmark & status \\
    \midrule
    $5$ & $43$ & $D_5$ & $195.94\ldots$ & known $43$-dim.\ packing $226.51\ldots$ exceeds this ceiling: no bare cross-section reaches it \\
    $6$ & $42$ & $E_6$ & $106.65\ldots$ & attained exactly; the section is optimal \\
    $7$ & $41$ & $E_7$ & $61.57\ldots$ & attained exactly; the section is optimal \\
    $8$ & $40$ & $E_8$ & $41.05\ldots$ & attained exactly; the section is optimal \\
    \midrule
    $9$ & $39$ & $\Lambda_9$ & $19.35\ldots$ & bare improvement possible; $\det K<19683$ would also improve the rank-$9$ lattice record \\
    $10$ & $38$ & $\Lambda_{10}$ & $10.53\ldots$ & bare improvement possible; $\det K\le44286$ would also improve the rank-$10$ lattice record \\
    \bottomrule
  \end{tabular}
  \caption{Bare cross-section bounds and current-record benchmarks.
    For $k\le8$ (above the rule), the displayed value is the unconditional
    ceiling of Proposition~\ref{prop:ceiling}, because the Hermite constant is
    known from published results. For $k=9,10$ we instead display the benchmark
    $\underline{\delta}^{\mathrm{lat}}_{k}(3/2)^{24-k}$ from the densest
    lattice currently known in dimension $k$ \cite{nebe2026density}. Exceeding
    this benchmark with a bare cross-section would simultaneously produce a
    new $k$-dimensional lattice record. A recent preprint, discussed in
    Section~\ref{sec:ceiling},
  claims that the $k=9$ benchmark is in fact unconditional.}
  \label{tab:ceiling}
\end{table}

For $k=9,10$ we do not use an established value of the Hermite constant, so
putting the current record $\underline{\delta}^{\mathrm{lat}}_{k}$ in place
of $\delta^{\mathrm{lat}}_{k}$ gives only a benchmark. Since
$\underline{\delta}^{\mathrm{lat}}_{k}\le\delta^{\mathrm{lat}}_{k}$, the
resulting determinant benchmark
$(3/2)^k/(\underline{\delta}^{\mathrm{lat}}_{k})^2$ is at least as large as
the true determinant floor, whereas the density benchmark
$\underline{\delta}^{\mathrm{lat}}_{k}(3/2)^{24-k}$ is at most as large as
the true ceiling. Thus crossing either benchmark would establish a new
rank-$k$ lattice record, but failing to cross it does not rule out a bare
improvement in dimensions $48-k$. Concretely, in dimension $39$ a section
with $19683\le\det K<21870$ would improve our packing without necessarily improving the
current rank-$9$ lattice record; a smaller determinant would improve both. In
dimension $38$, a section with $44287\le\det K<54675$ would improve our
packing without necessarily improving the current rank-$10$ lattice record,
whereas $\det K\le44286$ would improve both. The exact rank-$10$ benchmark is
$3^{11}/4=44286.75$, so the last inequality is equivalent because
$\det K$ is integral.

No bare cross-section, then, beats $195.94\ldots$ in dimension $43$. A cluster
does, and this is where the second lever in~\eqref{eq:master} earns its keep.

Let $K$ be the $5$-dimensional lattice with Gram matrix
\[
  \mathbf{K} \;=\;
  \begin{bmatrix}
    6 & -1 & -1 & -2 & -1 \\
    -1 &  6 & -2 & -1 & -1 \\
    -1 & -2 &  6 & -1 & -1 \\
    -2 & -1 & -1 &  6 & -1 \\
    -1 & -1 & -1 & -1 &  6
  \end{bmatrix},
  \qquad
  \mathbf{M} \;=\; \mathbf{K}^{-1} \;=\; \frac{1}{48}
  \begin{bmatrix}
    14 &  7 &  7 &  8 &  6 \\
    7 & 14 &  8 &  7 &  6 \\
    7 &  8 & 14 &  7 &  6 \\
    8 &  7 &  7 & 14 &  6 \\
    6 &  6 &  6 &  6 & 12
  \end{bmatrix},
\]
so $\det \mathbf{K} = 3072 = 2^{10}\cdot 3$. It is a saturated cross-section
of $P_{48n}$ and of $P_{48m}$ (Appendix~\ref{app:certificates}). The
hypotheses of Theorem~\ref{th:simplex} hold with $\beta = 7/24$: the diagonal
entries of $\mathbf{M}$ are $14/48$ and $12/48$, and the quantities
$\mathbf{M}_{i,i}+\mathbf{M}_{j,j}-2\mathbf{M}_{i,j}$ are all $12/48$ or
$14/48$. The cluster is $S = \{0, m_1, \ldots, m_5\}$, the origin together
with the basis of $M$ dual to the chosen basis of $K$; its six points form
three pairs at squared distance $1/4$, namely $\{0,m_5\}$, $\{m_1,m_4\}$
and $\{m_2,m_3\}$, and all other distances equal $\beta = 7/24$. Hence
\[
  \delta \;=\; 6 \cdot \frac{1}{\sqrt{3072}}
  \left( \frac{6-\tfrac7{24}}{4} \right)^{43/2}
  \;=\; \frac{137^{43/2}}{2^{223/2}\cdot 3^{21}}
  \;=\; 226.51\ldots,
\]
improving the 1982 record $195.94\ldots$ by a factor $1.15\ldots$ that
splits as $6 \cdot \sqrt{972/3072} \cdot (137/144)^{43/2}$: cluster size,
determinant penalty, radius penalty. The lattice $K$ is not a rescaled root
lattice ($3072/3^5 \notin \mathbb{Z}$), has $12$ minimal vectors, and has
discriminant group $\mathbb{Z}/4\times\mathbb{Z}/16\times\mathbb{Z}/48$.

\section{Kissing numbers}
\label{sec:kissing}

Proposition~\ref{prop:kissing} turns a kissing number into a finite count over
the complete list of $52{,}416{,}000$ minimal vectors of the ambient, so every
packing in this paper comes with one for free. Three families are worth
recording: the $1982$ cross-sections, whose kissing numbers appear never to
have been computed; the chain of Section~\ref{sec:chain}, which improves the
tabulated bounds in dimensions $40$, $41$ and $42$; and one further packing in
dimension $45$. Table~\ref{tab:kissing} collects what the three give in
dimensions $40$ to $47$.

\subsection{The 1982 cross-sections}
\label{sec:cor8}

Corollary~8 of Conway and Sloane~\cite{conway1982laminated} constructs
cross-sections of $P_{48p}$ orthogonal to
\[
  \sqrt3A_1,\ \sqrt3A_2,\ \sqrt3A_3,\ \sqrt3D_3,\ \sqrt3D_4,\ \ldots,\ \sqrt3D_7,\
  S_7,\ S_8,\ S_9,\ S_{10},
\]
where $S_n = \langle \sqrt3D_{n-1}, v^{*}\rangle$ are the glue lattices of the
corollary's proof. The resulting bare packings live in dimensions $47$ down to
$38$, with two in dimension $41$ (``$D_7$ or $S_7$'' in the corollary) and, as
the table shows, two in dimension $45$ as well. The corollary states their
densities but not their kissing numbers. We reconstruct all twelve in closed
form, in the coordinate frame of the corollary's own proof, and count exactly
(Table~\ref{tab:cor8}).

\begin{table}[ht]
  \centering
  \begin{tabular}{ccrr}
    \toprule
    dim & complement & kissing number & previously tabulated \\
    \midrule
    $47$ & $\sqrt3A_1$ & $\mathbf{23{,}766{,}960}$ & $9{,}741{,}412$ \\
    $46$ & $\sqrt3A_2$ & $\mathbf{12{,}309{,}600}$ & $5{,}318{,}060$ \\
    $45$ & $\sqrt3A_3$ & $6{,}687{,}648$ & $3{,}047{,}160$ \\
    $45$ & $\sqrt3D_3$ & $\mathbf{6{,}702{,}080}$ & $3{,}047{,}160$ \\
    $44$ & $\sqrt3D_4$ & $\mathbf{4{,}153{,}600}$ & $2{,}948{,}552$ \\
    $43$ & $\sqrt3D_5$ & $\mathbf{2{,}545{,}056}$ & $1{,}745{,}692$ \\
    $42$ & $\sqrt3D_6$ & $\mathbf{1{,}541{,}292}$ & $1{,}250{,}676$ \\
    $41$ & $\sqrt3D_7$ & $922{,}282$ & $1{,}170{,}384$ \\
    $41$ & $S_7$ & $922{,}664$ & $1{,}170{,}384$ \\
    $40$ & $S_8$ & $578{,}368$ & $1{,}064{,}368$ \\
    $39$ & $S_9$ & $360{,}716$ & $755{,}988$ \\
    $38$ & $S_{10}$ & $224{,}608$ & $566{,}652$ \\
    \bottomrule
  \end{tabular}
  \caption{The twelve cross-section packings of Conway and
    Sloane~\cite[Corollary~8]{conway1982laminated} and their kissing numbers,
    computed here. Bold entries exceed the previously tabulated lower bounds
    \cite{edel1998kissing, brouwer2026tables, cohn2026kissing}. The two
    dimension-$45$ rows are two embeddings of one lattice ($A_3 = D_3$), which
    the corollary's family contains both as the $A$-series and as the
  $D$-fork; see Remark~\ref{rem:embedding}.}
  \label{tab:cor8}
\end{table}

The count needs an explicit embedding, and the proof supplies one: it fixes
coordinates, taken from Leech and Sloane~\cite{leech1971sphere}, in which the
minimal vectors of $P_{48p}$ include all vectors with two entries $\pm3$ and
the rest zero, together with $v^{*} = (-2, 1^{14}, 0^{33})$. Our frame for
$P_{48p}$ is exactly that frame, so those vectors exist verbatim and a vector
of the shape of $v^{*}$ exists on the support of every weight-$15$ codeword.
We pin the lexicographically first such codeword and write every complement
basis down in closed form, with no search. Each kissing number is then one
sweep of the minimal shell, counting the fiber over zero
(Remark~\ref{rem:bare}).

In dimensions $42$ to $47$ these counts exceed the previously tabulated lower
bounds; in dimensions $38$ to $41$ they fall below them. The packings that
held the density records for four decades were also hiding the best kissing
numbers, in the range where nobody had looked.

\begin{remark}
  \label{rem:embedding}
  Unlike the density, which by Theorem~\ref{th:antipode} depends only on
  $\det L$, the kissing number of a cross-section packing depends on the
  embedding of the complement in the ambient lattice. The two dimension-$45$
  rows of Table~\ref{tab:cor8} show this concretely: $\sqrt3A_3$ and
  $\sqrt3D_3$ are the same lattice, but the corollary's family contains it
  twice, once as the $A$-series and once as the $D$-fork, and the two
  embeddings give $6{,}687{,}648$ and $6{,}702{,}080$ contacts. Since the
  corollary pins the embeddings only up to its frame condition, the values in
  Table~\ref{tab:cor8} are those of the canonical frame fixed above.
\end{remark}

\subsection{The chain}

The complements of $K_6$, $K_7$ and $K_8$ improve on
Table~\ref{tab:cor8} in their own dimensions, and on the tabulated bounds as
well:
\[
  1{,}792{,}386, \qquad 1{,}324{,}472, \qquad 1{,}092{,}000
\]
in dimensions $42$, $41$ and $40$. Only the first competes with a $1982$
section: in dimension $42$ the bound to beat is Table~\ref{tab:cor8}'s
$1{,}541{,}292$, while in dimensions $41$ and $40$ the $1982$ sections fall
well short of the tabulated $1{,}170{,}384$ and $1{,}064{,}368$, which are
therefore the relevant comparison. All three counts are the same in $P_{48p}$,
$P_{48q}$ and $P_{48n}$.

\begin{remark}
  \label{rem:lowkissing}
  The last two sections of the chain do not improve their dimensions. The
  complement of $K_9$ has $604{,}644$, $604{,}472$ and $604{,}892$ contacts in
  $P_{48p}$, $P_{48q}$ and $P_{48n}$, and that of $K_{10}$ has $381{,}658$,
  $381{,}438$ and $381{,}620$, all below the records $755{,}988$ and
  $566{,}652$. That the three counts differ is itself informative: by
  Theorem~\ref{th:antipode} the three embeddings of each section give packings
  of equal density, so they are inequivalent under the automorphism groups of
  their ambients even though no density computation can distinguish them.
\end{remark}

\subsection{A kissing record in dimension 45}
\label{sec:kissing45}

Ranking cross-sections by the kissing number of their antipode packing rather
than by density yields one more packing, which beats even the corresponding
$1982$ cross-section. Here $K$ is the $3$-dimensional cross-section of
$P_{48m}$ with Gram matrix
\[
  \mathbf{K} \;=\;
  \begin{bmatrix}
    6 & -2 & -2 \\
    -2 &  6 & -2 \\
    -2 & -2 &  6
  \end{bmatrix},
  \qquad
  \mathbf{M} \;=\; \mathbf{K}^{-1} \;=\; \frac{1}{8}
  \begin{bmatrix}
    2 & 1 & 1 \\
    1 & 2 & 1 \\
    1 & 1 & 2
  \end{bmatrix},
  \qquad \det \mathbf{K} = 128,
\]
isometric to $2\sqrt2A_3^{*}$, the rescaled body-centered cubic lattice. The
hypotheses of Theorem~\ref{th:simplex} hold with $\beta = \tfrac14$, all
pairwise squared distances of $S = \{0, m_1, m_2, m_3\}$ being exactly
$\tfrac14$, and the center density
\[
  \delta \;=\; 4 \cdot \frac{1}{\sqrt{128}}
  \left( \frac{6 - \tfrac14}{4} \right)^{45/2}
  \;=\; \frac{23^{45/2}}{2^{183/2}} \;=\; 1243.45\ldots
\]
equals the record of Chen, Hu, Li, Wang and Wu~\cite{chen2025antipode}. Every
fiber of Proposition~\ref{prop:kissing} contains $2{,}459{,}946$ minimal
vectors and the per-sphere counts are equal across translates, so every sphere
touches
\[
  3 \cdot 2{,}459{,}946 \;=\; 7{,}379{,}838
\]
others, against $6{,}702{,}080$ for the better of the two $1982$
dimension-$45$ cross-sections. The section embeds in $P_{48q}$ as well, giving
a packing of the same density by Theorem~\ref{th:antipode} and, although the
kissing number need not survive a change of embedding
(Remark~\ref{rem:embedding}), an exact count there returns the same
$7{,}379{,}838$.

\begin{table}[ht]
  \centering
  \begin{tabular}{crrr}
    \toprule
    dim & previously tabulated & CS82~\cite[Cor~8]{conway1982laminated} & this paper \\
    \midrule
    $40$ & $1{,}064{,}368$ & $578{,}368$ & $\mathbf{1{,}092{,}000}$ \\
    $41$ & $1{,}170{,}384$ & $922{,}664$ & $\mathbf{1{,}324{,}472}$ \\
    $42$ & $1{,}250{,}676$ & $1{,}541{,}292$ & $\mathbf{1{,}792{,}386}$ \\
    $43$ & $1{,}745{,}692$ & $\mathbf{2{,}545{,}056}$ & \\
    $44$ & $2{,}948{,}552$ & $\mathbf{4{,}153{,}600}$ & \\
    $45$ & $3{,}047{,}160$ & $6{,}702{,}080$ & $\mathbf{7{,}379{,}838}$ \\
    $46$ & $5{,}318{,}060$ & $\mathbf{12{,}309{,}600}$ & \\
    $47$ & $9{,}741{,}412$ & $\mathbf{23{,}766{,}960}$ & \\
    \bottomrule
  \end{tabular}
  \caption{Kissing number lower bounds in dimensions $40$ to $47$. The middle
    column is Table~\ref{tab:cor8}, the kissing numbers of the $1982$
    cross-sections; the last is this paper's chain (dimensions $40$ to $42$)
    and the dimension-$45$ packing of Section~\ref{sec:kissing45}. Bold marks
    the largest entry of each row, and every one of them improves on the
    previously tabulated bound \cite{edel1998kissing, brouwer2026tables,
    cohn2026kissing}. For comparison the best known upper bounds in dimensions
    $44$ to $47$ are $220{,}788{,}272$, $316{,}735{,}249$, $441{,}900{,}184$
  and $621{,}658{,}419$ \cite{delaat2024solving, cohn2026kissing}.}
  \label{tab:kissing}
\end{table}

\section{Search and verification}
\label{sec:search}

Everything that decides a claim in this paper is computed in exact integer or
rational arithmetic; floating point and lattice reduction appear only inside
search heuristics.

The seed of Section~\ref{sec:seed} needs no search: $\sigma$ is written down
in closed form on the coordinate frame of the ternary code, its fixed lattice
is an integer kernel computation, and the $480$ minimal vectors of that
fixed lattice split into two orthogonal root systems of type $E_8$, from
which a basis with Gram matrix $3C(E_8)$ is read off.

The chain of Section~\ref{sec:chain} was found by sweeping the complete
list of $52{,}416{,}000$ minimal vectors of the ambient once, classifying each
by its projection onto the span of the current section
(Remark~\ref{rem:e8component}), and ranking extensions by the exact Schur
complement of Proposition~\ref{prop:extend}(2). The optimality of each step
is not a search result: $\rho^2_{\mathrm{glue}}(K)$ is computed by enumerating the $\det K$
classes of $K^{*}/K$ ($2187$, $4374$, $6561$, $21870$ and $54675$ classes
respectively) and taking, for each class, the minimum norm of its
representatives, a finite computation of rank at most $10$. As a
check, the dual shortest-vector problem of
Proposition~\ref{prop:extend}(3) was also solved directly in the complement
for the step $\sqrt3E_8 \to K_9$, with the enumeration running to completion,
and returned the same value $\mu(L^{*}) = 10/3$.

The dimension-$43$ section of Section~\ref{sec:clusters} was found by
simulated annealing over $5$-tuples of ambient lattice vectors: a move
replaces or perturbs one vector of the tuple, each tuple is completed to the
saturated section its span meets, and the section is scored by an exact lower
bound for the best antipode density over its dual; the best candidates are
then re-scored exhaustively, enumerating all short dual points and running a
maximum-clique search over admissible pairs for each realizable diameter.
Independent runs over $P_{48n}$ and over $P_{48m}$ converged to isometric
sections.

Each packing is recorded as a machine-checkable certificate that names the
ambient lattice, the section basis in ambient coordinates and the cluster,
and a verifier rebuilds the ambient from its published
construction and re-proves membership, saturation, $\det L = \det K$, the
cluster distances, the minimum, and the density. The code, the verifier and
certificates for every packing in this paper, in each of the ambients
where it exists, are available at
\url{https://github.com/wcgbg/sphere-packing}.

\section{Open problems}
\label{sec:open}

Table~\ref{tab:ceiling} localizes what is left for the bare
cross-section route and separates unconditional ceilings from current-record
benchmarks. In dimensions $40$, $41$ and $42$ the bare route is closed: the
sections of Corollary~\ref{cor:optimal} are optimal among bare
cross-section packings. This does not rule out antipode clusters over these or
other sections. In dimension $43$ the bare route is also closed, while the
cluster route is essential for our improvement and may admit further gains.

The interesting bare cases are dimensions $39$ and $38$. In dimension $39$,
any rank-$9$ section with determinant below $21870$ would improve our present
packing. The current published rank-$9$ lattice record corresponds to the
benchmark $19683$: under the published-literature convention used here, a
determinant below that is not excluded by an optimality theorem, but would
simultaneously produce a new $9$-dimensional lattice record. Thus
$[19683,21870)$ is the portion of the search interval that could improve
dimension $39$ without necessarily improving that lower-dimensional record. If the recent
dimension-$9$ preprint cited above is correct, then determinants below $19683$
are excluded and this interval is the full rigorous search window. Likewise,
for dimension $38$ one seeks determinant below $54675$. The current rank-$10$
lattice record gives the exact benchmark $3^{11}/4=44286.75$; because section
determinants are integral, $\det K\le44286$ would also improve the rank-$10$
lattice record, while $44287\le\det K<54675$ could improve dimension $38$
without doing so.

Any such section inside an even ambient lattice must itself be an even
integral lattice of minimum at least $6$, which is a real constraint: at the
rank-$9$ benchmark $19683=3^9$ the natural candidate, the rescaled laminated
lattice $\sqrt{3/2}\,\Lambda_9$, is not integral, so a rank-$9$ section of
that determinant would have to be some other lattice. We do not know which
even integral lattices of minimum $6$ occur in ranks $9$ and $10$, and
determining their smallest determinant would sharpen these searches. The
current lattice records alone do not provide rigorous determinant floors in
these ranks; the recent dimension-$9$ preprint, if correct, supplies the
rank-$9$ floor. We leave the corresponding integral-section questions open.

\section*{AI Disclosure}

We used Claude and ChatGPT to assist with
background research, generating and exploring ideas, developing the
search and verification code, checking derivations, running and
analyzing the computations, and drafting and editing this paper. All
computational claims underlying the new packings and kissing-number
computations are independently established by the exact-arithmetic
certificates and verifier of Section~\ref{sec:search}. The theoretical and
literature-based claims are supported by the arguments and references in the
paper. The authors have reviewed all content and take full responsibility for
it.

\appendix

\section{Certificates}
\label{app:certificates}

Each packing is specified by integer matrices: $\mathbf{G}$, the
$48\times48$ Gram matrix of a basis $b_1,\ldots,b_{48}$ of the ambient
lattice $\Lambda$; $\mathbf{T}$, a $k\times48$ matrix whose row $i$ holds the
coordinates of the $i$-th basis vector of the cross-section with respect to
$b_1,\ldots,b_{48}$; and, when the cluster is not trivial, $\mathbf{W}$, an
$s\times48$ matrix whose row $i$ holds the coordinates of a lattice point
$w_i$ with $\pi_U(w_i) = u_i$ (row $0$ being $w_0 = 0$). Identifying
$\Lambda$ with $\mathbb{Z}^{48}$ under
$\langle x,y\rangle = x\mathbf{G}y^{\mathsf{T}}$, the following
certificate-local claims can be checked in exact arithmetic:
\begin{enumerate}
  \item $\mathbf{G}$ is symmetric and positive definite with even diagonal
    and $\det\mathbf{G}=1$, so $\Lambda$ is an even unimodular lattice of
    rank $48$; positive definiteness can be checked in exact arithmetic, for
    example by an exact $LDL^{\mathsf{T}}$ decomposition;
  \item $\mathbf{T}\mathbf{G}\mathbf{T}^{\mathsf{T}} = \mathbf{K}$, the Gram
    matrix displayed in the text, with the stated determinant;
  \item the greatest common divisor of the $k\times k$ minors of $\mathbf{T}$
    is $1$ (equivalently, the Smith normal form of $\mathbf{T}$ is
    $[\,I_k\mid 0\,]$), so the rows of $\mathbf{T}$ span a saturated
    sublattice: $K = \Lambda \cap U$ exactly;
  \item $\mathbf{W}\mathbf{G}\mathbf{T}^{\mathsf{T}}$ is a zero row above
    $I_k$, which is the statement $S = \{0, m_1,\ldots,m_k\}$ of
    Theorem~\ref{th:simplex};
  \item the pairwise squared distances
    $d^2(u_i,u_j) = (c_i-c_j)\mathbf{K}^{-1}(c_i-c_j)^{\mathsf{T}}$ are as
    listed, all at most $\beta < 6$, with $\beta$ attained;
  \item Theorem~\ref{th:antipode} with $\det L = \det\mathbf{K}$ gives the
    stated density.
\end{enumerate}
The packing itself is determined by the same data:
$L = \{x \in \mathbb{Z}^{48} : x\mathbf{G}\mathbf{T}^{\mathsf{T}} = 0\}$ is
an integer kernel, and
$\mathcal{A}(S) = \bigcup_i (w_i - \pi_U(w_i) + L)$.

One claim is not certified by the printed data: the minimum $\mu = 6$ of the
ambient. For $P_{48p}$, $\mathbf{G}$ is the Gram matrix of the basis built
from the Pless symmetry code \cite{pless1972symmetry, conway1999sphere}; for
$P_{48m}$ it is the catalogue entry \texttt{P48m}. In both cases $\mu = 6$ is
the published extremality of the ambient, and a reader who prefers not to rely
on it can re-verify $\mu = 6$ from $\mathbf{G}$ by a short-vector enumeration.
Certificates for the same packings in the other
ambients ($P_{48q}$ and $P_{48n}$ for dimensions $38$--$42$, $P_{48n}$ for
dimension $43$, $P_{48q}$ for dimension $45$) are in the repository.

\subsection{Dimensions 42, 41 and 40}

Certificates \\
\texttt{records/dim42-delta106/P48p-e6f899b4}, \\
\texttt{records/dim41-delta61/P48p-28a12f16}, and \\
\texttt{records/dim40-delta41/P48p-6a1733b0}.

The three sections are the row
prefixes of one matrix: rows $1$--$6$ span $K_6 \cong \sqrt3E_6$, rows
$1$--$7$ span $K_7 \cong \sqrt3E_7$, and all eight rows span
$K_8 \cong \sqrt3E_8$, with
$\mathbf{T}\mathbf{G}\mathbf{T}^{\mathsf{T}} = 3\,C(E_8)$.

\paragraph{$\mathbf{G}$} (the Gram matrix of $P_{48p}$):
\begin{Verbatim}[fontsize=\tiny,breaklines=true,breakanywhere=true]
[[6 3 3 -3 3 -3 2 -3 -2 -3 3 3 -1 1 -1 -3 -2 -1 -2 -2 0 1 2 -1 0 0 2 2 -1 0 -2 -1 -1 0 -2 1 -2 1 -1 -2 0 3 -3 -3 -1 0 2 -2]
 [3 6 3 0 3 -3 0 -1 1 -3 3 3 1 0 0 0 -3 -2 -2 -2 -2 0 3 1 2 -2 -1 2 0 0 -1 0 1 0 -1 0 0 -1 -2 1 2 2 -1 -2 -2 -2 0 -1]
 [3 3 6 -3 3 -3 -1 -3 0 -3 3 3 2 -1 -2 0 -3 -2 -3 -2 0 2 3 2 0 1 0 0 1 2 -1 1 1 -2 1 -1 0 0 -1 -1 1 3 -3 -3 1 0 1 0]
 [-3 0 -3 6 0 0 1 2 0 2 0 0 1 -1 3 1 1 0 1 0 0 -3 -1 -1 0 -2 -2 1 -1 -2 1 -1 1 0 1 -1 1 -2 1 1 1 -2 2 2 -2 -1 -2 2]
 [3 3 3 0 6 -3 2 -3 -2 -3 3 3 1 -2 1 0 -1 -3 -2 -3 1 0 1 1 -1 0 -1 2 0 -1 0 0 0 0 1 -1 0 -2 -2 -2 1 3 -1 -2 0 -2 1 0]
 [-3 -3 -3 0 -3 6 0 1 2 1 -3 -3 0 -1 0 1 2 0 2 3 0 1 -2 -1 -1 0 1 -1 2 0 0 -1 1 2 1 2 1 -1 0 1 -2 -2 2 2 2 2 1 -1]
 [2 0 -1 1 2 0 6 0 -1 0 2 1 0 -1 0 -1 1 -2 1 -2 0 -2 -2 -2 -2 0 2 1 -2 -1 -1 -3 1 2 -1 1 -2 -1 1 -1 -2 1 0 0 0 0 0 -2]
 [-3 -1 -3 2 -3 1 0 6 1 3 -2 -3 -1 2 1 1 0 2 1 0 -1 -1 -1 0 2 0 0 0 -1 1 0 1 0 0 -1 -1 -1 1 1 2 -1 -1 3 1 0 0 -3 0]
 [-2 1 0 0 -2 2 -1 1 6 -1 0 -1 2 0 -1 1 1 -1 2 0 -2 1 1 1 1 0 -1 -2 1 1 0 0 3 1 0 0 0 -1 0 2 1 -1 0 2 1 -1 0 0]
 [-3 -3 -3 2 -3 1 0 3 -1 6 -3 -3 -2 0 1 2 1 3 2 2 -1 -2 -3 -1 -1 0 1 0 -2 1 2 0 0 -1 0 -1 1 0 2 2 -2 -3 2 1 -1 2 -2 0]
 [3 3 3 0 3 -3 2 -2 0 -3 6 3 1 -1 0 -1 -1 -2 -2 -3 1 -1 2 1 0 -1 0 1 -1 0 -2 -2 1 -1 0 0 -2 0 0 -1 2 3 -2 -1 0 -1 1 -1]
 [3 3 3 0 3 -3 1 -3 -1 -3 3 6 1 0 0 -2 -1 -2 -2 -1 1 0 2 -1 -1 -1 0 0 1 -1 -1 -1 0 -1 0 0 -1 0 -1 -2 2 1 -2 -2 -1 -1 1 1]
 [-1 1 2 1 1 0 0 -1 2 -2 1 1 6 -1 -2 1 -1 -3 0 -1 0 0 1 1 0 0 -1 -1 1 -1 -1 1 2 0 1 -1 1 -2 1 0 1 0 -1 1 0 0 0 1]
 [1 0 -1 -1 -2 -1 -1 2 0 0 -1 0 -1 6 -1 -2 -1 3 1 1 0 1 1 -2 1 1 1 -1 0 -1 -2 2 -2 -1 -2 -1 -3 2 1 -1 0 0 -1 -1 -2 0 -1 0]
 [-1 0 -2 3 1 0 0 1 -1 1 0 0 -2 -1 6 -1 2 1 0 0 2 0 0 -1 0 -1 -1 3 -1 0 1 -1 -1 0 1 0 0 -1 -2 0 2 -1 2 0 0 -1 0 1]
 [-3 0 0 1 0 1 -1 1 1 2 -1 -2 1 -2 -1 6 -1 -1 1 1 -2 -2 -1 3 0 -1 -1 -1 1 0 1 1 1 0 2 -1 3 -2 1 3 -1 0 2 1 0 0 -1 0]
 [-2 -3 -3 1 -1 2 1 0 1 1 -1 -1 -1 -1 2 -1 6 0 3 1 1 0 -3 -1 -2 1 -1 -1 0 0 2 -1 0 1 0 1 -1 0 0 -1 0 -2 1 3 1 -1 1 0]
 [-1 -2 -2 0 -3 0 -2 2 -1 3 -2 -2 -3 3 1 -1 0 6 0 2 0 0 -1 -1 1 1 0 0 -1 1 1 1 -2 -2 0 -1 -1 1 1 0 0 -2 1 -1 -1 1 -2 0]
 [-2 -2 -3 1 -2 2 1 1 2 2 -2 -2 0 1 0 1 3 0 6 1 -1 -1 -2 -2 -1 0 0 -1 0 -1 0 -1 0 2 -1 0 0 0 2 1 -1 -2 0 2 -1 0 0 -1]
 [-2 -2 -2 0 -3 3 -2 0 0 2 -3 -1 -1 1 0 1 1 2 1 6 -1 1 -2 -1 -1 0 1 -1 1 -1 0 1 0 0 1 2 1 0 0 1 -1 -2 1 1 -1 1 0 0]
 [0 -2 0 0 1 0 0 -1 -2 -1 1 1 0 0 2 -2 1 0 -1 -1 6 1 1 -1 -2 1 1 0 0 -1 0 0 -2 -1 1 -1 -1 0 -1 -3 1 0 0 0 2 1 1 1]
 [1 0 2 -3 0 1 -2 -1 1 -2 -1 0 0 1 0 -2 0 0 -1 1 1 6 1 0 -1 2 1 0 1 2 -1 2 0 -1 1 0 -1 0 -2 -1 0 1 -1 -1 2 0 1 0]
 [2 3 3 -1 1 -2 -2 -1 1 -3 2 2 1 1 0 -1 -3 -1 -2 -2 1 1 6 1 2 -1 0 0 1 0 -1 1 0 -1 0 -1 0 1 -1 0 2 2 -2 -2 0 0 1 1]
 [-1 1 2 -1 1 -1 -2 0 1 -1 1 -1 1 -2 -1 3 -1 -1 -2 -1 -1 0 1 6 2 0 -2 -1 0 1 1 2 1 -1 1 -1 2 0 0 2 1 2 1 0 2 -2 0 0]
 [0 2 0 0 -1 -1 -2 2 1 -1 0 -1 0 1 0 0 -2 1 -1 -1 -2 -1 2 2 6 -1 -2 1 0 0 -1 1 -1 1 -2 0 1 2 0 2 2 1 1 -1 -1 -1 -1 0]
 [0 -2 1 -2 0 0 0 0 0 0 -1 -1 0 1 -1 -1 1 1 0 0 1 2 -1 0 -1 6 0 -2 0 1 1 2 0 0 0 0 -1 0 0 -2 -2 1 -1 -1 2 0 -1 0]
 [2 -1 0 -2 -1 1 2 0 -1 1 0 0 -1 1 -1 -1 -1 0 0 1 1 1 0 -2 -2 0 6 0 -1 0 -2 -1 0 0 -1 1 -1 1 1 0 -2 0 0 -1 0 3 1 -2]
 [2 2 0 1 2 -1 1 0 -2 0 1 0 -1 -1 3 -1 -1 0 -1 -1 0 0 0 -1 1 -2 0 6 -2 0 -1 -1 -1 0 -1 0 0 -1 -2 0 1 1 0 -1 -1 0 1 -1]
 [-1 0 1 -1 0 2 -2 -1 1 -2 -1 1 1 0 -1 1 0 -1 0 1 0 1 1 0 0 0 -1 -2 6 0 0 0 0 0 2 0 1 0 0 -1 0 0 0 -1 1 1 1 1]
 [0 0 2 -2 -1 0 -1 1 1 1 0 -1 -1 -1 0 0 0 1 -1 -1 -1 2 0 1 0 1 0 0 0 6 1 0 1 -2 1 0 -1 1 0 1 -1 0 0 -1 2 1 0 -1]
 [-2 -1 -1 1 0 0 -1 0 0 2 -2 -1 -1 -2 1 1 2 1 0 0 0 -1 -1 1 -1 1 -2 -1 0 1 6 1 1 0 1 -1 2 -1 -1 0 -1 -2 1 1 1 -1 -1 1]
 [-1 0 1 -1 0 -1 -3 1 0 0 -2 -1 1 2 -1 1 -1 1 -1 1 0 2 1 2 1 2 -1 -1 0 0 1 6 -1 -2 0 -2 0 0 -1 0 0 1 0 0 0 -1 -1 1]
 [-1 1 1 1 0 1 1 0 3 0 1 0 2 -2 -1 1 0 -2 0 0 -2 0 0 1 -1 0 0 -1 0 1 1 -1 6 0 1 0 1 -1 0 1 -1 0 -1 1 1 -1 0 0]
 [0 0 -2 0 0 2 2 0 1 -1 -1 -1 0 -1 0 0 1 -2 2 0 -1 -1 -1 -1 1 0 0 0 0 -2 0 -2 0 6 -1 3 1 0 -1 1 -1 0 1 1 0 -1 0 -1]
 [-2 -1 1 1 1 1 -1 -1 0 0 0 0 1 -2 1 2 0 0 -1 1 1 1 0 1 -2 0 -1 -1 2 1 1 0 1 -1 6 -1 1 -2 0 0 0 0 1 0 1 0 -1 2]
 [1 0 -1 -1 -1 2 1 -1 0 -1 0 0 -1 -1 0 -1 1 -1 0 2 -1 0 -1 -1 0 0 1 0 0 0 -1 -2 0 3 -1 6 0 1 -1 1 -1 0 0 1 0 0 1 -2]
 [-2 0 0 1 0 1 -2 -1 0 1 -2 -1 1 -3 0 3 -1 -1 0 1 -1 -1 0 2 1 -1 -1 0 1 -1 2 0 1 1 1 0 6 -1 0 2 0 -1 1 0 0 0 0 1]
 [1 -1 0 -2 -2 -1 -1 1 -1 0 0 0 -2 2 -1 -2 0 1 0 0 0 0 1 0 2 0 1 -1 0 1 -1 0 -1 0 -2 1 -1 6 1 0 0 1 -1 -1 0 0 1 0]
 [-1 -2 -1 1 -2 0 1 1 0 2 0 -1 1 1 -2 1 0 1 2 0 -1 -2 -1 0 0 0 1 -2 0 0 -1 -1 0 -1 0 -1 0 1 6 1 -1 -1 0 0 -1 2 -1 0]
 [-2 1 -1 1 -2 1 -1 2 2 2 -1 -2 0 -1 0 3 -1 0 1 1 -3 -1 0 2 2 -2 0 0 -1 1 0 0 1 1 0 1 2 0 1 6 0 -1 2 1 -1 0 -1 -1]
 [0 2 1 1 1 -2 -2 -1 1 -2 2 2 1 0 2 -1 0 0 -1 -1 1 0 2 1 2 -2 -2 1 0 -1 -1 0 -1 -1 0 -1 0 0 -1 0 6 -1 0 0 -1 -2 0 2]
 [3 2 3 -2 3 -2 1 -1 -1 -3 3 1 0 0 -1 0 -2 -2 -2 -2 0 1 2 2 1 1 0 1 0 0 -2 1 0 0 0 0 -1 1 -1 -1 -1 6 -2 -2 1 -1 1 -1]
 [-3 -1 -3 2 -1 2 0 3 0 2 -2 -2 -1 -1 2 2 1 1 0 1 0 -1 -2 1 1 -1 0 0 0 0 1 0 -1 1 1 0 1 -1 0 2 0 -2 6 1 0 0 -2 0]
 [-3 -2 -3 2 -2 2 0 1 2 1 -1 -2 1 -1 0 1 3 -1 2 1 0 -1 -2 0 -1 -1 -1 -1 -1 -1 1 0 1 1 0 1 0 -1 0 1 0 -2 1 6 0 -1 0 1]
 [-1 -2 1 -2 0 2 0 0 1 -1 0 -1 0 -2 0 0 1 -1 -1 -1 2 2 0 2 -1 2 0 -1 1 2 1 0 1 0 1 0 0 0 -1 -1 -1 1 0 0 6 0 1 0]
 [0 -2 0 -1 -2 2 0 0 -1 2 -1 -1 0 0 -1 0 -1 1 0 1 1 0 0 -2 -1 0 3 0 1 1 -1 -1 -1 -1 0 0 0 0 2 0 -2 -1 0 -1 0 6 1 -1]
 [2 0 1 -2 1 1 0 -3 0 -2 1 1 0 -1 0 -1 1 -2 0 0 1 1 1 0 -1 -1 1 1 1 0 -1 -1 0 0 -1 1 0 1 -1 -1 0 1 -2 0 1 1 6 -1]
 [-2 -1 0 2 0 -1 -2 0 0 0 -1 1 1 0 1 0 0 0 -1 0 1 0 1 0 0 0 -2 -1 1 -1 1 1 0 -1 2 -2 1 0 0 -1 2 -1 0 1 0 -1 -1 6]]
\end{Verbatim}

\paragraph{$\mathbf{T}$} ($8\times48$):
\begin{Verbatim}[fontsize=\tiny,breaklines=true,breakanywhere=true]
[[5 -6 2 1 0 1 -4 3 1 0 1 1 -1 0 0 -1 -1 -2 0 0 0 -1 -1 -1 -1 -1 -1 0 0 0 1 1 1 1 0 0 0 -2 1 1 1 1 1 -1 -1 -1 0 -1]
 [-1 -2 -3 -2 1 0 0 2 0 -1 2 2 0 -1 0 1 -1 0 0 0 -1 1 1 0 0 1 0 1 0 0 1 0 0 0 0 0 0 0 0 0 0 -1 -1 0 -1 0 0 0]
 [-5 4 1 0 -2 -4 7 -3 0 -3 -3 -3 0 -1 1 1 1 2 -1 2 1 1 1 1 0 0 0 0 1 0 -1 -1 -1 -1 -1 -1 0 2 -1 0 -1 -1 -1 1 0 1 0 0]
 [-5 3 0 0 0 0 2 -4 0 2 -1 -1 0 3 0 -1 0 -1 -1 0 0 0 0 1 1 0 1 0 0 1 0 -1 -1 0 0 0 0 0 0 -1 0 1 0 0 0 0 0 0]
 [11 -4 -2 0 1 3 -8 7 -1 0 4 3 1 -4 -1 1 0 1 3 -2 -1 0 -1 -2 -1 0 -1 -1 -1 -2 0 2 2 0 1 1 0 -1 0 1 0 -1 1 -1 1 0 0 1]
 [1 -2 6 2 1 2 -2 0 1 0 -3 -1 -1 3 0 -2 1 -2 -1 0 1 -3 -1 -1 0 -1 0 0 -1 0 1 -1 0 0 0 0 0 -1 1 1 1 2 1 0 0 -1 0 -1]
 [-7 4 -4 -3 -3 -3 6 -5 -1 0 1 0 0 0 2 3 -1 1 -1 -1 -1 3 1 1 1 1 0 1 1 0 -1 1 0 0 0 0 0 2 0 -2 -1 -2 -1 1 0 1 0 0]
 [-1 -2 3 3 3 -2 2 1 3 0 -1 -3 -1 0 -2 -5 -1 0 0 5 2 -2 1 2 -2 -1 0 0 2 3 -1 -1 -2 1 -2 -1 1 -1 -1 1 0 1 0 0 -1 -1 0 0]]
\end{Verbatim}

\subsection{Dimension 39}

Certificate \texttt{records/dim39-delta18/P48p-adc3124f}; the ambient is
$P_{48p}$ and $\mathbf{T}\mathbf{G}\mathbf{T}^{\mathsf{T}} = \mathbf{K}_9$.

\paragraph{$\mathbf{T}$} ($9\times48$):
\begin{Verbatim}[fontsize=\tiny,breaklines=true,breakanywhere=true]
[[5 -6 2 1 0 1 -4 3 1 0 1 1 -1 0 0 -1 -1 -2 0 0 0 -1 -1 -1 -1 -1 -1 0 0 0 1 1 1 1 0 0 0 -2 1 1 1 1 1 -1 -1 -1 0 -1]
 [-5 4 1 0 -2 -4 7 -3 0 -3 -3 -3 0 -1 1 1 1 2 -1 2 1 1 1 1 0 0 0 0 1 0 -1 -1 -1 -1 -1 -1 0 2 -1 0 -1 -1 -1 1 0 1 0 0]
 [-10 7 1 0 -2 -4 9 -7 0 -1 -4 -4 0 2 1 0 1 1 -2 2 1 1 1 2 1 0 1 0 1 1 -1 -2 -2 -1 -1 -1 0 2 -1 -1 -1 0 -1 1 0 1 0 0]
 [0 1 -4 -2 0 -1 1 2 -1 -2 2 1 1 -3 0 2 0 2 1 0 -1 2 1 0 0 1 0 0 0 -1 0 0 0 -1 0 0 0 1 -1 0 -1 -2 -1 0 0 1 0 1]
 [-11 5 -2 -2 -1 -4 9 -5 0 -2 -2 -2 0 1 1 1 0 1 -2 2 0 2 2 2 1 1 1 1 1 1 0 -2 -2 -1 -1 -1 0 2 -1 -1 -1 -1 -2 1 -1 1 0 0]
 [1 -1 2 0 1 1 -1 2 0 -2 -1 0 0 0 0 0 1 0 0 0 0 -1 0 -1 0 0 0 0 -1 -1 1 -1 0 -1 0 0 0 0 0 1 0 0 0 0 0 0 0 0]
 [-6 3 -2 -3 -2 -2 5 -3 -1 -2 0 0 0 0 2 3 0 1 -1 -1 -1 2 1 0 1 1 0 1 0 -1 0 0 0 -1 0 0 0 2 0 -1 -1 -2 -1 1 0 1 0 0]
 [5 -8 6 1 0 0 -3 4 2 -3 1 -1 -2 0 1 -2 -1 -2 0 1 0 -2 -1 -2 -2 -2 -2 0 0 -1 1 1 1 0 -1 0 1 -2 1 2 0 0 2 -1 -1 -1 0 -1]
 [-12 3 7 0 -2 -4 10 -7 1 -3 -4 -6 -2 5 2 -2 0 -1 -3 3 1 -1 1 1 0 -1 0 1 1 1 0 -2 -2 -1 -2 -1 1 1 0 0 -1 0 0 1 -1 0 0 -1]]
\end{Verbatim}

\subsection{Dimension 38}

Certificate \texttt{records/dim38-delta9/P48p-a17f250f}; the ambient is
$P_{48p}$ and $\mathbf{T}\mathbf{G}\mathbf{T}^{\mathsf{T}} = \mathbf{K}_{10}$.

\paragraph{$\mathbf{T}$} ($10\times48$):
\begin{Verbatim}[fontsize=\tiny,breaklines=true,breakanywhere=true]
[[5 -6 2 1 0 1 -4 3 1 0 1 1 -1 0 0 -1 -1 -2 0 0 0 -1 -1 -1 -1 -1 -1 0 0 0 1 1 1 1 0 0 0 -2 1 1 1 1 1 -1 -1 -1 0 -1]
 [-5 4 1 0 -2 -4 7 -3 0 -3 -3 -3 0 -1 1 1 1 2 -1 2 1 1 1 1 0 0 0 0 1 0 -1 -1 -1 -1 -1 -1 0 2 -1 0 -1 -1 -1 1 0 1 0 0]
 [-10 7 1 0 -2 -4 9 -7 0 -1 -4 -4 0 2 1 0 1 1 -2 2 1 1 1 2 1 0 1 0 1 1 -1 -2 -2 -1 -1 -1 0 2 -1 -1 -1 0 -1 1 0 1 0 0]
 [0 1 -4 -2 0 -1 1 2 -1 -2 2 1 1 -3 0 2 0 2 1 0 -1 2 1 0 0 1 0 0 0 -1 0 0 0 -1 0 0 0 1 -1 0 -1 -2 -1 0 0 1 0 1]
 [-11 5 -2 -2 -1 -4 9 -5 0 -2 -2 -2 0 1 1 1 0 1 -2 2 0 2 2 2 1 1 1 1 1 1 0 -2 -2 -1 -1 -1 0 2 -1 -1 -1 -1 -2 1 -1 1 0 0]
 [1 -1 2 0 1 1 -1 2 0 -2 -1 0 0 0 0 0 1 0 0 0 0 -1 0 -1 0 0 0 0 -1 -1 1 -1 0 -1 0 0 0 0 0 1 0 0 0 0 0 0 0 0]
 [-6 3 -2 -3 -2 -2 5 -3 -1 -2 0 0 0 0 2 3 0 1 -1 -1 -1 2 1 0 1 1 0 1 0 -1 0 0 0 -1 0 0 0 2 0 -1 -1 -2 -1 1 0 1 0 0]
 [5 -8 6 1 0 0 -3 4 2 -3 1 -1 -2 0 1 -2 -1 -2 0 1 0 -2 -1 -2 -2 -2 -2 0 0 -1 1 1 1 0 -1 0 1 -2 1 2 0 0 2 -1 -1 -1 0 -1]
 [-12 3 7 0 -2 -4 10 -7 1 -3 -4 -6 -2 5 2 -2 0 -1 -3 3 1 -1 1 1 0 -1 0 1 1 1 0 -2 -2 -1 -2 -1 1 1 0 0 -1 0 0 1 -1 0 0 -1]
 [-2 1 -5 -2 0 -2 1 -1 -1 0 1 0 1 -2 0 0 -1 1 0 1 -1 1 2 0 -1 1 0 0 1 1 -1 0 0 0 -1 -1 0 0 0 0 0 0 0 0 0 0 0 0]]
\end{Verbatim}

\subsection{Dimension 43}

Certificate \texttt{records/dim43-delta226/P48m-1f02996f}. The ambient is
$P_{48m}$; $\mathbf{T}\mathbf{G}\mathbf{T}^{\mathsf{T}} = \mathbf{K}$ is the
Gram matrix of Section~\ref{sec:clusters} and
$\mathbf{W}\mathbf{G}\mathbf{T}^{\mathsf{T}} = [\,0; I_5\,]$. Its Gram matrix
$\mathbf{G}$ is the catalogue entry \texttt{P48m} \cite{nebe2014fourth,
nebe2026catalogue} reproduced verbatim; the dimension-$45$ packing below uses
the same ambient and the same $\mathbf{G}$.

\paragraph{$\mathbf{G}$} (the Gram matrix of $P_{48m}$):
\begin{Verbatim}[fontsize=\tiny,breaklines=true,breakanywhere=true]
[[6 -2 1 2 -1 -1 -3 1 1 -1 2 -3 2 0 1 -2 -3 -2 2 -2 -3 -2 -1 1 0 2 1 0 -2 2 -2 -2 3 2 1 -2 1 3 3 2 -3 3 0 -3 -2 0 2 -2]
 [-2 8 -1 -2 -2 1 0 -3 -2 3 2 2 2 -2 3 0 -2 3 -4 0 4 4 -3 1 2 1 2 0 2 0 -1 3 -4 -1 -4 2 0 -4 -2 -3 1 0 0 -1 3 3 -4 3]
 [1 -1 6 -2 0 -3 0 2 -2 2 -1 -3 -2 1 -2 2 -3 -2 -2 2 -3 -1 -1 0 -1 2 1 2 -3 0 1 -1 0 -1 -1 0 0 0 -2 2 2 -2 2 -1 0 -1 2 2]
 [2 -2 -2 6 1 2 -3 1 0 -3 0 1 1 1 2 -3 1 -2 4 -2 0 -3 1 2 -2 1 0 -1 2 1 1 -2 0 2 2 -2 -1 1 1 0 -1 3 -1 -1 -1 -2 2 -1]
 [-1 -2 0 1 6 -2 -2 2 -2 -2 1 2 -3 -1 -1 0 2 -3 0 1 -1 -3 2 -2 -3 0 -1 2 1 -2 -1 0 0 -2 3 -1 -3 -2 1 2 1 -2 2 2 -2 -3 -1 -2]
 [-1 1 -3 2 -2 6 2 0 1 0 -1 2 1 2 1 0 1 1 1 -2 3 2 0 2 1 -1 1 -2 2 0 1 1 0 0 -2 0 0 1 1 -2 1 2 -3 -1 0 1 1 1]
 [-3 0 0 -3 -2 2 8 -1 2 3 -1 1 0 1 -3 4 2 2 -2 2 3 3 2 1 1 -2 -2 1 -1 -1 0 0 2 -2 -2 0 2 0 -2 0 0 -2 -1 2 0 0 1 0]
 [1 -3 2 1 2 0 -1 6 -1 0 -1 0 -1 0 -1 -1 0 -2 2 1 -2 -3 0 1 -2 -1 -1 0 -1 -1 0 0 1 0 2 -2 0 1 0 1 3 -2 2 0 -3 -3 2 0]
 [1 -2 -2 0 -2 1 2 -1 6 -2 -1 0 2 1 -1 0 1 2 3 0 1 2 2 1 0 -1 -2 0 0 2 -1 0 3 2 2 -1 1 4 2 1 -3 1 -1 0 -1 2 2 -3]
 [-1 3 2 -3 -2 0 3 0 -2 6 1 0 0 0 -1 1 -1 1 -3 2 1 1 -2 1 1 0 1 1 -2 0 0 0 0 -2 -4 1 2 -2 -2 -1 1 -1 1 -1 1 1 -1 3]
 [2 2 -1 0 1 -1 -1 -1 -1 1 6 1 2 -3 2 -1 -1 0 -1 -1 0 0 0 0 1 0 0 2 -1 1 -3 0 0 -1 -1 0 0 -2 1 1 -3 2 1 -1 -1 0 -3 -1]
 [-3 2 -3 1 2 2 1 0 0 0 1 6 0 -1 1 0 3 1 0 0 4 1 2 0 -1 -2 -2 0 3 -1 0 1 -1 -1 0 0 -1 -3 -1 -1 1 -1 0 2 0 -1 -2 0]
 [2 2 -2 1 -3 1 0 -1 2 0 2 0 6 -1 2 -2 -1 2 1 -2 1 2 0 3 1 -1 -1 -1 0 1 -1 0 0 2 0 0 2 1 0 -1 -2 3 0 -1 0 2 -1 -1]
 [0 -2 1 1 -1 2 1 0 1 0 -3 -1 -1 6 -3 1 0 -2 1 0 0 0 0 1 -1 1 2 0 0 1 2 -2 1 0 0 0 0 2 1 1 1 2 -2 -1 0 2 3 1]
 [1 3 -2 2 -1 1 -3 -1 -1 -1 2 1 2 -3 6 -2 -1 2 0 -2 1 1 -2 1 1 1 0 -1 2 0 0 2 -3 1 -1 0 0 -1 0 -2 -1 1 0 -1 1 0 -2 1]
 [-2 0 2 -3 0 0 4 -1 0 1 -1 0 -2 1 -2 6 0 0 -3 1 0 3 1 -1 0 0 0 1 0 -2 0 0 1 -2 -2 1 0 0 -1 1 1 -2 -1 1 1 0 1 1]
 [-3 -2 -3 1 2 1 2 0 1 -1 -1 3 -1 0 -1 0 6 1 2 1 1 -1 3 -1 -1 -2 -3 -1 1 -1 1 0 1 -1 2 0 0 -1 0 -1 0 -1 0 3 0 -1 -1 -2]
 [-2 3 -2 -2 -3 1 2 -2 2 1 0 1 2 -2 2 0 1 6 0 1 2 4 -1 1 2 -1 -1 -1 1 0 0 3 -1 1 -1 2 2 0 -1 -2 -1 -1 0 1 2 2 -2 1]
 [2 -4 -2 4 0 1 -2 2 3 -3 -1 0 1 1 0 -3 2 0 8 0 -2 -2 1 2 -2 0 -1 -1 1 3 0 -1 2 4 3 -1 1 3 1 0 -1 3 -1 0 -1 -1 2 -2]
 [-2 0 2 -2 1 -2 2 1 0 2 -1 0 -2 0 -2 1 1 1 0 6 1 0 0 1 -1 1 -1 3 0 0 0 0 0 -1 1 -1 0 -1 -2 1 1 -3 2 1 0 0 1 0]
 [-3 4 -3 0 -1 3 3 -2 1 1 0 4 1 0 1 0 1 2 -2 1 8 2 0 1 1 -1 -1 0 3 0 0 1 -2 -1 -1 -1 0 -2 -2 -2 0 -1 -1 1 0 1 -1 0]
 [-2 4 -1 -3 -3 2 3 -3 2 1 0 1 2 0 1 3 -1 4 -2 0 2 8 -1 1 2 0 0 0 2 0 -1 3 -1 0 -3 3 1 0 0 -1 0 0 -2 0 2 4 -1 2]
 [-1 -3 -1 1 2 0 2 0 2 -2 0 2 0 0 -2 1 3 -1 1 0 0 -1 6 0 -1 -2 -3 1 0 -1 1 -2 2 -1 2 -1 -2 0 0 2 -1 0 0 2 -1 -2 1 -3]
 [1 1 0 2 -2 2 1 1 1 1 0 0 3 1 1 -1 -1 1 2 1 1 1 0 6 -1 1 0 1 1 1 0 0 0 1 0 -1 1 1 -1 0 0 2 0 -1 0 0 2 1]
 [0 2 -1 -2 -3 1 1 -2 0 1 1 -1 1 -1 1 0 -1 2 -2 -1 1 2 -1 -1 6 -1 1 -2 -1 1 0 0 -1 0 -2 1 1 -1 0 -1 -1 1 -2 -1 1 2 -1 1]
 [2 1 2 1 0 -1 -2 -1 -1 0 0 -2 -1 1 1 0 -2 -1 0 1 -1 0 -2 1 -1 6 2 2 0 1 0 0 0 0 -1 0 -1 0 1 1 -1 1 0 -2 1 1 1 0]
 [1 2 1 0 -1 1 -2 -1 -2 1 0 -2 -1 2 0 0 -3 -1 -1 -1 -1 0 -3 0 1 2 6 0 0 1 0 0 -1 0 -2 1 -1 0 1 0 1 2 -1 -3 1 2 0 2]
 [0 0 2 -1 2 -2 1 0 0 1 2 0 -1 0 -1 1 -1 -1 -1 3 0 0 1 1 -2 2 0 6 -1 0 -1 0 0 -2 0 -1 -2 -1 0 3 -1 -1 2 0 -1 0 0 -1]
 [-2 2 -3 2 1 2 -1 -1 0 -2 -1 3 0 0 2 0 1 1 1 0 3 2 0 1 -1 0 0 -1 6 -1 0 1 -2 1 1 0 -1 -1 0 -1 1 0 -2 1 1 0 0 1]
 [2 0 0 1 -2 0 -1 -1 2 0 1 -1 1 1 0 -2 -1 0 3 0 0 0 -1 1 1 1 1 0 -1 6 -1 -1 0 2 0 0 1 1 0 0 -2 3 -1 -2 0 2 1 0]
 [-2 -1 1 1 -1 1 0 0 -1 0 -3 0 -1 2 0 0 1 0 0 0 0 -1 1 0 0 0 0 -1 0 -1 6 -1 -2 0 -1 1 -1 0 -2 -1 2 0 0 0 2 0 1 2]
 [-2 3 -1 -2 0 1 0 0 0 0 0 1 0 -2 2 0 0 3 -1 0 1 3 -2 0 0 0 0 0 1 -1 -1 6 -2 -1 -1 2 0 -1 0 -2 2 -2 1 1 1 1 -3 1]
 [3 -4 0 0 0 0 2 1 3 0 0 -1 0 1 -3 1 1 -1 2 0 -2 -1 2 0 -1 0 -1 0 -2 0 -2 -2 8 0 1 -2 1 4 4 2 -2 1 -1 -1 -2 -1 3 -3]
 [2 -1 -1 2 -2 0 -2 0 2 -2 -1 -1 2 0 1 -2 -1 1 4 -1 -1 0 -1 1 0 0 0 -2 1 2 0 -1 0 6 1 0 2 2 0 0 -1 2 -1 0 0 0 2 0]
 [1 -4 -1 2 3 -2 -2 2 2 -4 -1 0 0 0 -1 -2 2 -1 3 1 -1 -3 2 0 -2 -1 -2 0 1 0 -1 -1 1 1 8 -3 -1 2 1 2 -1 -1 2 2 -2 -1 1 -4]
 [-2 2 0 -2 -1 0 0 -2 -1 1 0 0 0 0 0 1 0 2 -1 -1 -1 3 -1 -1 1 0 1 -1 0 0 1 2 -2 0 -3 6 1 -1 -1 -1 1 0 -1 1 2 2 -2 2]
 [1 0 0 -1 -3 0 2 0 1 2 0 -1 2 0 0 0 0 2 1 0 0 1 -2 1 1 -1 -1 -2 -1 1 -1 0 1 2 -1 1 6 1 -1 -1 -1 0 0 1 0 0 0 1]
 [3 -4 0 1 -2 1 0 1 4 -2 -2 -3 1 2 -1 0 -1 0 3 -1 -2 0 0 1 -1 0 0 -1 -1 1 0 -1 4 2 2 -1 1 8 3 1 -2 1 -1 -2 -1 1 4 -2]
 [3 -2 -2 1 1 1 -2 0 2 -2 1 -1 0 1 0 -1 0 -1 1 -2 -2 0 0 -1 0 1 1 0 0 0 -2 0 4 0 1 -1 -1 3 8 2 -3 2 -1 -2 -2 2 1 -2]
 [2 -3 2 0 2 -2 0 1 1 -1 1 -1 -1 1 -2 1 -1 -2 0 1 -2 -1 2 0 -1 1 0 3 -1 0 -1 -2 2 0 2 -1 -1 1 2 6 -2 -1 1 0 -3 -1 3 -2]
 [-3 1 2 -1 1 1 0 3 -3 1 -3 1 -2 1 -1 1 0 -1 -1 1 0 0 -1 0 -1 -1 1 -1 1 -2 2 2 -2 -1 -1 1 -1 -2 -3 -2 8 -2 0 1 1 -1 0 3]
 [3 0 -2 3 -2 2 -2 -2 1 -1 2 -1 3 2 1 -2 -1 -1 3 -3 -1 0 0 2 1 1 2 -1 0 3 0 -2 1 2 -1 0 0 1 2 -1 -2 8 -3 -3 1 2 0 0]
 [0 0 2 -1 2 -3 -1 2 -1 1 1 0 0 -2 0 -1 0 0 -1 2 -1 -2 0 0 -2 0 -1 2 -2 -1 0 1 -1 -1 2 -1 0 -1 -1 1 0 -3 6 1 -1 -1 -2 -1]
 [-3 -1 -1 -1 2 -1 2 0 0 -1 -1 2 -1 -1 -1 1 3 1 0 1 1 0 2 -1 -1 -2 -3 0 1 -2 0 1 -1 0 2 1 1 -2 -2 0 1 -3 1 6 0 -2 -2 -1]
 [-2 3 0 -1 -2 0 0 -3 -1 1 -1 0 0 0 1 1 0 2 -1 0 0 2 -1 0 1 1 1 -1 1 0 2 1 -2 0 -2 2 0 -1 -2 -3 1 1 -1 0 6 2 -2 3]
 [0 3 -1 -2 -3 1 0 -3 2 1 0 -1 2 2 0 0 -1 2 -1 0 1 4 -2 0 2 1 2 0 0 2 0 1 -1 0 -1 2 0 1 2 -1 -1 2 -1 -2 2 8 -1 0]
 [2 -4 2 2 -1 1 1 2 2 -1 -3 -2 -1 3 -2 1 -1 -2 2 1 -1 -1 1 2 -1 1 0 0 0 1 1 -3 3 2 1 -2 0 4 1 3 0 0 -2 -2 -2 -1 8 0]
 [-2 3 2 -1 -2 1 0 0 -3 3 -1 0 -1 1 1 1 -2 1 -2 0 0 2 -3 1 1 0 2 -1 1 0 2 1 -3 0 -4 2 1 -2 -2 -2 3 0 -1 -1 3 0 0 8]]
\end{Verbatim}

\paragraph{$\mathbf{T}$} ($5\times48$):
\begin{Verbatim}[fontsize=\tiny,breaklines=true,breakanywhere=true]
[[-1 1 0 0 0 0 0 0 -1 0 0 -1 0 1 1 0 0 -1 0 1 0 0 1 0 0 -1 0 0 0 0 0 1 1 1 0 0 0 0 0 0 -1 0 0 0 0 0 0 0]
 [-1 -2 7 -2 5 2 -2 -6 1 4 1 2 3 -1 -1 0 1 0 4 -2 1 0 0 1 2 1 1 -1 -1 -2 0 1 -2 1 1 -2 0 1 1 -2 0 -2 -2 0 -1 0 1 0]
 [-1 0 -7 1 -3 -3 -1 6 0 -1 3 -5 -5 1 0 0 0 1 -2 -2 2 4 1 1 -3 0 0 0 -1 1 1 -2 1 0 0 0 1 -2 -1 1 1 0 1 -1 2 0 0 -1]
 [2 -2 -2 1 1 0 1 -2 -3 0 -4 3 2 0 -1 2 -3 3 0 1 -1 -3 1 -1 1 1 -2 0 0 1 -2 1 -1 -1 -1 0 0 1 0 1 1 1 1 -1 -1 1 -1 1]
 [-1 3 -8 1 -8 -2 -3 7 -2 -3 1 -5 -6 0 -1 1 2 -1 -3 -2 2 4 0 1 -3 -1 -1 3 -1 2 2 -2 3 0 1 1 1 -2 -1 2 0 0 1 -1 1 -1 -1 -1]]
\end{Verbatim}

\paragraph{$\mathbf{W}$} ($6\times48$):
\begin{Verbatim}[fontsize=\tiny,breaklines=true,breakanywhere=true]
[[0 0 0 0 0 0 0 0 0 0 0 0 0 0 0 0 0 0 0 0 0 0 0 0 0 0 0 0 0 0 0 0 0 0 0 0 0 0 0 0 0 0 0 0 0 0 0 0]
 [9 7 1 0 8 6 0 -5 2 4 -4 -2 2 -2 -2 3 1 0 5 0 -1 -1 -3 -1 2 -2 -4 2 0 0 2 -1 -1 -1 0 1 -2 1 0 1 1 -1 1 1 0 -1 0 1]
 [6 8 8 -1 12 8 -1 -8 2 6 -6 0 5 -2 1 3 2 -1 7 0 -1 -3 -2 -2 4 -2 -3 2 0 -1 0 0 -1 0 0 0 -1 3 0 1 0 0 0 1 -1 0 0 1]
 [8 6 2 0 6 5 0 -3 4 3 -2 -2 1 -1 0 2 2 -1 4 1 0 0 -2 -1 1 -2 -2 1 0 -1 1 -2 -1 -1 0 2 -2 0 0 1 1 -1 1 0 1 -1 0 1]
 [8 9 -1 2 6 3 1 -2 2 4 -4 -4 1 0 -1 4 2 0 5 -1 -1 -1 -2 -2 2 -2 -4 2 0 0 1 -1 -1 -1 0 1 -2 1 0 1 1 -1 1 1 0 -1 0 1]
 [8 4 -2 1 4 4 -2 -2 1 2 -3 -2 2 -2 -4 4 2 1 2 0 0 0 -3 0 1 -1 -3 2 0 1 2 0 -1 -1 0 1 -2 0 0 1 1 -1 1 1 0 -1 0 1]]
\end{Verbatim}

\subsection{Dimension 45}

Certificate \texttt{records/dim45-kissing7379838/P48m-d914e7d8}. The ambient is
again $P_{48m}$, so $\mathbf{G}$ is the matrix printed above. As in dimension
$43$, the basis encoded by $\mathbf{T}$ is the one whose Gram matrix is the
$\mathbf{K}$ of Section~\ref{sec:kissing45}, so
$\mathbf{W}\mathbf{G}\mathbf{T}^{\mathsf{T}} = [\,0; I_3\,]$: the cluster is
$\{0\}$ together with the dual basis $m_1, m_2, m_3$, and no separate cluster
matrix is needed.

\paragraph{$\mathbf{T}$} ($3\times48$):
\begin{Verbatim}[fontsize=\tiny,breaklines=true,breakanywhere=true]
[[4 0 -6 0 1 -1 1 0 0 -2 0 0 -1 1 -1 1 -2 2 -1 1 -1 0 -1 1 -1 0 -1 -1 -1 1 1 -1 0 -1 -1 0 -1 -1 -1 0 1 -1 1 0 0 0 0 1]
 [2 0 6 -1 2 2 3 -3 1 0 -2 4 3 -1 1 -2 1 -2 1 2 -2 -2 -1 -2 2 0 1 0 2 -1 -1 1 -1 0 0 1 -1 2 0 -1 -1 1 0 1 -1 0 0 1]
 [-3 0 -3 2 -3 0 -1 2 0 -1 1 -1 -1 0 -1 1 -2 0 -2 1 -1 0 0 0 -1 0 -1 0 0 1 1 0 1 0 0 0 1 -1 0 1 0 0 0 -1 0 0 -1 -1]]
\end{Verbatim}

\paragraph{$\mathbf{W}$} ($4\times48$):
\begin{Verbatim}[fontsize=\tiny,breaklines=true,breakanywhere=true]
[[0 0 0 0 0 0 0 0 0 0 0 0 0 0 0 0 0 0 0 0 0 0 0 0 0 0 0 0 0 0 0 0 0 0 0 0 0 0 0 0 0 0 0 0 0 0 0 0]
 [1 1 -7 8 2 -7 0 5 1 4 -4 -4 1 2 1 5 -1 1 0 -1 3 3 0 -2 2 -1 0 1 -2 2 1 1 1 -1 0 1 -1 -1 0 2 1 0 1 0 1 -1 -1 0]
 [0 3 -1 6 5 -3 2 4 2 4 -4 -4 1 3 3 3 0 -1 -1 -1 1 0 0 -3 4 -2 -1 3 2 2 0 2 3 1 0 2 -1 0 -1 1 -1 1 1 0 0 0 -2 0]
 [1 5 -10 12 -1 -10 4 9 0 3 -6 -6 -3 3 2 5 -3 -1 -2 -1 2 1 1 -3 2 -3 -2 3 1 3 1 2 4 0 0 3 -1 -1 -1 3 -1 2 2 -1 1 0 -3 0]]
\end{Verbatim}

\bibliographystyle{plain}
\bibliography{refs}

@article{conway1982laminated,
  title   = {Laminated lattices},
  author  = {Conway, John H. and Sloane, Neil J. A.},
  journal = {Annals of Mathematics},
  series  = {2},
  volume  = {116},
  number  = {3},
  pages   = {593--620},
  year    = {1982},
  doi     = {10.2307/2007025}
}

@article{conway1996antipode,
  title     = {The antipode construction for sphere packings},
  author    = {Conway, John H. and Sloane, Neil J. A.},
  journal   = {Inventiones mathematicae},
  volume    = {123},
  pages     = {309--313},
  year      = {1996},
  publisher = {Springer}
}

@book{conway1999sphere,
  title     = {Sphere Packings, Lattices and Groups},
  author    = {Conway, John H. and Sloane, Neil J. A.},
  year      = {1999},
  publisher = {Springer-Verlag},
  address   = {New York},
  series    = {Grundlehren der mathematischen Wissenschaften},
  volume    = {290},
  edition   = {3rd},
  isbn      = {978-0-387-98585-5},
  doi       = {10.1007/978-1-4757-6568-7}
}

@misc{chen2025antipode,
  title        = {New sphere packings from the antipode construction},
  author       = {Chen, Ruitao and Hu, Jiachen and Li, Binghui and Wang, Liwei and Wu, Tianyi},
  year         = {2025},
  howpublished = {Preprint, arXiv:2505.02394}
}

@article{hales2005proof,
  title   = {A proof of the {K}epler conjecture},
  author  = {Hales, Thomas C.},
  journal = {Annals of Mathematics},
  volume  = {162},
  number  = {3},
  pages   = {1065--1185},
  year    = {2005}
}

@article{viazovska2017sphere,
  title   = {The sphere packing problem in dimension 8},
  author  = {Viazovska, Maryna S.},
  journal = {Annals of Mathematics},
  volume  = {185},
  number  = {3},
  pages   = {991--1015},
  year    = {2017}
}

@article{cohn2017sphere,
  title   = {The sphere packing problem in dimension 24},
  author  = {Cohn, Henry and Kumar, Abhinav and Miller, Stephen D. and Radchenko, Danylo and Viazovska, Maryna},
  journal = {Annals of Mathematics},
  volume  = {185},
  number  = {3},
  pages   = {1017--1033},
  year    = {2017}
}

@article{pless1972symmetry,
  title   = {Symmetry codes over {$GF(3)$} and new five-designs},
  author  = {Pless, Vera},
  journal = {Journal of Combinatorial Theory, Series A},
  volume  = {12},
  number  = {1},
  pages   = {119--142},
  year    = {1972}
}

@article{nebe2014fourth,
  title   = {A fourth extremal even unimodular lattice of dimension 48},
  author  = {Nebe, Gabriele},
  journal = {Discrete Mathematics},
  volume  = {331},
  pages   = {133--136},
  year    = {2014},
  doi     = {10.1016/j.disc.2014.05.011}
}

@misc{nebe2026catalogue,
  title        = {A Catalogue of Lattices},
  author       = {Nebe, Gabriele and Sloane, Neil J. A.},
  year         = {2026},
  howpublished = {\url{https://www.math.rwth-aachen.de/~Gabriele.Nebe/LATTICES/}},
  note         = {Accessed July 2026}
}

@misc{cohn2026packing,
  title        = {Sphere packing},
  author       = {Cohn, Henry},
  year         = {2026},
  howpublished = {Online table of the best sphere packings known, \url{https://cohn.mit.edu/sphere-packing/}},
  note         = {Accessed July 2026}
}

@article{leech1971sphere,
  author  = {Leech, John and Sloane, N. J. A.},
  title   = {Sphere packings and error-correcting codes},
  journal = {Canadian Journal of Mathematics},
  volume  = {23},
  number  = {4},
  year    = {1971},
  pages   = {718--745},
}

@article{nebe2013automorphisms,
  title   = {On automorphisms of extremal even unimodular lattices of dimension 48},
  author  = {Nebe, Gabriele},
  journal = {International Journal of Number Theory},
  volume  = {9},
  number  = {8},
  pages   = {1933--1959},
  year    = {2013},
  doi     = {10.1142/S179304211350067X}
}

@book{martinet2003perfect,
  title     = {Perfect Lattices in Euclidean Spaces},
  author    = {Martinet, Jacques},
  year      = {2003},
  publisher = {Springer-Verlag},
  address   = {Berlin},
  series    = {Grundlehren der mathematischen Wissenschaften},
  volume    = {327},
  doi       = {10.1007/978-3-662-05167-2}
}

@article{blichfeldt1935minimum,
  title   = {The minimum values of positive quadratic forms in six, seven and eight variables},
  author  = {Blichfeldt, Hans Frederick},
  journal = {Mathematische Zeitschrift},
  volume  = {39},
  pages   = {1--15},
  year    = {1935}
}

@article{korkine1877formes,
  title   = {Sur les formes quadratiques positives},
  author  = {Korkine, Alexander and Zolotareff, Yegor},
  journal = {Mathematische Annalen},
  volume  = {11},
  pages   = {242--292},
  year    = {1877}
}

@misc{nebe2026density,
  title        = {Table of Densest Packings Presently Known},
  author       = {Nebe, Gabriele and Sloane, Neil J. A.},
  year         = {2026},
  howpublished = {Part of the Catalogue of Lattices, \url{https://www.math.rwth-aachen.de/~Gabriele.Nebe/LATTICES/density.html}},
  note         = {Accessed August 2026}
}

@article{rains1998shadow,
  title   = {The shadow theory of modular and unimodular lattices},
  author  = {Rains, Eric M. and Sloane, Neil J. A.},
  journal = {Journal of Number Theory},
  volume  = {73},
  pages   = {359--389},
  year    = {1998}
}

@article{hohn2016fixed,
  author  = {H{\"o}hn, Gerald and Mason, Geoffrey},
  title   = {The 290 fixed-point sublattices of the {L}eech lattice},
  journal = {Journal of Algebra},
  volume  = {448},
  pages   = {618--637},
  year    = {2016},
  doi     = {10.1016/j.jalgebra.2015.09.033},
}

@article{nikulin1980integral,
  author  = {Nikulin, V. V.},
  title   = {Integral symmetric bilinear forms and some of their applications},
  journal = {Mathematics of the USSR-Izvestiya},
  volume  = {14},
  number  = {1},
  pages   = {103--167},
  year    = {1980},
  doi     = {10.1070/IM1980v014n01ABEH001060},
}

@misc{cohn2026kissing,
  title        = {Kissing numbers},
  author       = {Cohn, Henry},
  year         = {2026},
  howpublished = {Online table of the best kissing numbers known, \url{https://cohn.mit.edu/kissing-numbers/}},
  note         = {Accessed July 2026}
}

@misc{brouwer2026tables,
  title        = {Bounds for binary constant weight codes},
  author       = {Brouwer, Andries E.},
  year         = {2026},
  howpublished = {Online tables, \url{https://www.win.tue.nl/~aeb/codes/Andw.html}},
  note         = {Accessed July 2026}
}

@article{edel1998kissing,
  author  = {Edel, Yves and Rains, Eric M. and Sloane, Neil J. A.},
  title   = {On kissing numbers in dimensions 32 to 128},
  journal = {Electronic Journal of Combinatorics},
  volume  = {5},
  pages   = {Research Paper 22},
  year    = {1998},
  doi     = {10.37236/1360},
}

@article{delaat2024solving,
  author  = {de Laat, David and Leijenhorst, Nando},
  title   = {Solving clustered low-rank semidefinite programs arising from polynomial optimization},
  journal = {Mathematical Programming Computation},
  volume  = {16},
  number  = {3},
  pages   = {503--534},
  year    = {2024},
  doi     = {10.1007/s12532-024-00264-w},
}

@misc{dutoursikiric2025dimension9,
  title        = {The lattice packing problem in dimension 9 by {Voronoi}'s algorithm},
  author       = {Dutour Sikiri{\'c}, Mathieu and van Woerden, Wessel},
  year         = {2025},
  howpublished = {Preprint, arXiv:2508.20719}
}

\end{document}